\documentclass[final,12pt]{colt2026} 

\title[Asymptotically optimal sequential change detection for bounded means]{Asymptotically optimal sequential change detection\\ for bounded means}
\usepackage{times}
\usepackage{comment}
\newcommand{\KL}{\mathrm{KL}}
\newcommand{\KLinf}{\mathrm{KL_{\mathrm{inf}}}}
\newcommand{\1}{\mathbbm{1}}
\newcommand{\EE}{\mathbb{E}}
\newcommand{\PP}{\mathbb{P}}
\newcommand{\RR}{\mathbb{R}}
\newcommand{\NN}{\mathbb{N}}

\usepackage{bbm}
\usepackage{enumitem}
\usepackage{mathrsfs}
\usepackage{comment}

\coltauthor{\Name{Ashwin Ram} \Email{aram2@andrew.cmu.edu}\\ 
\Name{Aaditya Ramdas} \Email{aramdas@cmu.edu}\\
 \addr Carnegie Mellon University}

\begin{document}

\maketitle

\begin{abstract}%
We consider the problem of quickest changepoint detection under the Average Run Length (ARL) constraint where the pre-change and post-change laws lie in composite families $\mathcal P$ and $\mathcal Q$ respectively. In such a problem, a massive challenge is characterizing the best possible detection delay when the ``hardest'' pre-change law in $\mathcal P$ depends on the unknown post-change law $Q\in\mathcal Q$. And typical simple-hypothesis likelihood-ratio arguments for Page-CUSUM and Shiryaev-Roberts do not at all apply here. To that end, we derive a universal sharp lower bound in full generality for any ARL-calibrated changepoint detector in the low type-I error ($\gamma\to\infty$ regime) of the order $\log(\gamma)/\KLinf(Q, \mathcal P)$. We show achievability of this universal lower bound by proving a tight matching upper bound (with the same sharp $\log\gamma$ constant) in the important bounded mean detection setting. In addition, for separated mean shifts, we also we derive a uniform minimax guarantee of this achievability over the alternatives.
\end{abstract}


\newcommand{\esssup}{\operatorname*{ess\,sup}}\newcommand{\CADD}{\mathcal{C}}\newcommand{\Lorden}{\mathcal{D}}
\section{Introduction}
Consider the following typical setting in sequential analysis: we observe a data stream, a change occurs, we detect it, and then we try to build detectors that will make this ``delay'' short. In his 1954 ``continuous inspection'' work \citep{Page1954}, Page formalized that very idea. Put simply, we accumulate more and more evidence that the data no longer come from the baseline distribution and then consequently stop when this evidence crosses a threshold. Such stopping rules were made as optimal-stopping objects so to speak by Shiryaev's Bayesian formulation \citep{Shiryaev1961, Shiryaev1963}. On the other hand, Lorden \citep{Lorden1971} and Pollak \citep{Pollak1985} gave us non Bayesian worst-case delay criteria, which are still prominent in modern-day streaming applications \citep{BassevilleNikiforov1993, PoorHadjiliadis2008, TartakovskyNikiforovBasseville2014}. Under the classic independent and identically distributed (i.i.d) model with a known pre-change law $f$ and known post-change law $g$, Moustakides proved that Page's CUSUM procedure is in fact exactly optimal for Lorden's criterion \citep{Moustakides1986}. And Pollak's CADD led to much better versions of the Shiryaev-Roberts (SR) procedure that obtain almost minimax performance \citep{Pollak1985, Pollak2009, PolunchenkoTartakovsky2010}. 

However, in practice assuming that the pre-change law and post-change law are single distributions is extremely oversimplified. For instance, sensor drift can make the pre-change law as general as any distribution with some property (e.g., bounded mean $\le m$), while the anomalies can be ``anything'' violating that property (e.g., mean $>m$). Especially even in such broad situations, practitioners may still want a stopping rule $T$ that guarantees a particular false-alarm rate. That issue is the core motivation of this paper, the setting of which we formalize at the set level. We take the no-change law to be i.i.d.\ $P$ for some unknown $P\in\mathcal{P}$ whereas given a change at time $k$, the data switches distribution to be i.i.d.\ $Q\in\mathcal Q$. We define the false alarm constraint as an ARL lower bound
\begin{equation}\label{eq:arl}
\inf_{P\in\mathcal P} \EE_{P^\infty}[T]\ge \gamma,    
\end{equation}
and the post-change delay is measured by Pollak's conditional average delay to detection,
\begin{equation}\label{eq:cadd}
\CADD_Q(T) := \sup_{k\ge 1} \EE_{k,Q}\left[(T-k+1)^+\ \middle|\ T\ge k\right], 
\end{equation}
which is always at most Lorden's worst-case conditional delay \citep{Lorden1971, Pollak1985}. However, when measuring information in these composite settings, we no longer use the basic $\KL(g\|f)$, but rather the least-favorable separation,
\begin{equation}\label{eq:klinfgap}
I(Q;\mathcal{P})\ :=\ \inf_{P\in\mathcal{P}} \KL(Q\|P).   
\end{equation}
Note that this $\KLinf$ is widely used in many domains, from robust hypothesis testing to distributionally robust QCD where inside an ambguity set, a least-favorable post-change law minimizes $\KL(\cdot\|P_0)$ \citep{MolloyFord2017, XieLiangVeeravalli2024}. With all this being said, our work distinguishes itself with several key contributions. First, we provide a fully general and pointwise information lower bound for composite pre-change classes. That is, for any ARL-calibrated family $\{T_\gamma\}$ and any alternative $Q$, we have that,
\begin{equation}\label{eq:mainlb}
\liminf_{\gamma\to\infty}\ \frac{\CADD_Q(T_\gamma)}{\log\gamma}\ \ge\ \frac{1}{I(Q;\mathcal{P})}.   
\end{equation}
In short, we prove the bound by first identifying a near least-favorable pre-change law $P_\delta$ that satisfies $\KL(Q\|P_\delta)\le I(Q;\mathcal{P})+\delta$. We then use a block argument (which is a consequence of \eqref{eq:arl}) to guarantee that we can find at least one window of time where the conditional null probability of stopping is at most $f/\gamma$ (where $f$ is the length of each block). We then show that under $Q$, the probability of stopping in that particular ``window'' is very small unless the log-likelihood random walk accumulates an atypical amount of evidence. But that's ruled out by the maximal strong law as we will see. More on all this later. Secondly, we provide a tight upper bound for the bounded mean model, by building a bounded mean detector that achieves this equality in \eqref{eq:mainlb}. That is, for every $Q\in\mathcal{Q}_m$ (letting $Q_m$ be the class of bounded alternatives with mean $>m$, and the nulls are those with mean $\le m$), we have that $\lim_{\gamma\to\infty}\frac{\CADD_Q(T^{\mathrm{BM}}_\gamma)}{\log\gamma} = \frac{1}{\mathrm{KL}_{\inf}(Q;m)}$. Here, $T^{\mathrm{BM}}_\gamma$ is the class of bounded mean ARL-calibrated detectors that we will detail. Finally, for separated alternatives, we obtain a uniform minimax constant $1$, showing sharpness in that case also.

\subsection*{Related Work}
Changepoint detection is often seen as starting from Page's CUSUM \citep{Page1954}, optimal for Lorden's criteria of worst-case conditional delay (when ($f,g$) are known) \citep{Moustakides1986, Ritov1990}. \citet{Lorden1971} proposed the criteria of $\sup_k\esssup\,\EE_{k}[(T-k+1)^+\mid\mathcal{F}_{k-1}]$, which is in fact so strict that even in situations where the changepoint is adversarially aligned with the sample path, it forces a detector to be uniformly fast. \citet{Pollak1985} proposed the CADD criteria \eqref{eq:cadd}, conditioning on survival until $k$ and naturally leading to SR-based stopping rules. There has since been a lot of work on SR-head-starting and initialization in a quasi-stationary way to approach minimax performance (and improve higher-order asymptotics) \citep{Pollak2009, PolunchenkoTartakovsky2010, TartakovskyPollakPolunchenko2012}. Moreover, there has been numbers of works that together present a beautiful unification of all these results, and give connections to Markovian extensions and renewal theory \citep{BassevilleNikiforov1993, PoorHadjiliadis2008, TartakovskyNikiforovBasseville2014}.

Looking at composite and robust changepoint detection, when the post-change parameters are unknown, that often will (and does in fact) necessitate generalized likelihood ratios or mixture detectors. \citet{Lai1995} established first order optimality when the window size grows slowly, \citet{Mei2006} studied unknown pre-change and post-change distribution parameters and derived asymptotically optimal procedures in exponential families, \citet{XieSiegmund2013} used mixture procedures in a multi-sensor setting, combining stream-wise generalized likelihood ratios. In decentralized communication, \citet{HadjiliadisZhangPoor2009} show that when each sensor runs CUSUM and transmits once (ie, one-shot communication), this can be asymptotically optimal with respect to Lorden. Even in non iid settings, there exist weighted and mixture SR procedures for composite post change hypotheses \citep{TartakovskyVeeravalli2005, PergamenchtchikovTartakovsky2019}. Interestingly, Huber-Strassen's capacity-based generalization of Neyman-Pearson seems to be the origin for least-favorable distributions (and minimax tests) \citep{HuberStrassen1973}. Today QCD that is robust to distributions will often involve amiguity sets (eg, Wasserstein balls that are around empirical post change samples), and in such cases the least-favorable post-change distribution is one minimizing $\KL(\cdot\|P_0)$ giving a ``robust'' CUSUM \citep{MolloyFord2017, XieLiangVeeravalli2024}.

The idea that betting-type nonnegative supermartingales can be seen as sequential evidence really starts from Ville \citep{Ville1939}, and was formalized later by \citet{Shafer2011}. \citet{Howard2020} developed the the idea of time-uniform boundary crossing via nonnegative supermartingales, while \citet{RufLarssonKoolenRamdas2023} generalized Ville's theorem to composite nulls, allowing us to do nonparametric anytime-valid inference. There has also been great amount work done on e-processes, with respect to surveying univeral inference operations with e-processes to merging sequential e-values into a single e-process, to providing specific optimality notions like growth rate optimality to deal with composite hypotheses \citep{RamdasWang2025, VovkWang2024, GrunwaldDeHeideKoolen2024}. One  crucial work in this area is the idea of e-detectors, which are sums of e-processes started consecutive times that provide nonasymptotic ARL control and nearly optimal detection delays for a wide range of nonparametric problems \citep{shin2024edetectors}. We've also seen very important reductions that connect sequential estimation to sequential changepoint detection \cite{ShekharRamdas2023}. With all this being said, we are the first work to prove such a lower bound like \eqref{eq:mainlb} holding for any composite $\mathcal{P}$ without restriction in the changepoint setting. In addition, we give a sharp exact $\log\gamma$ constant in the bounded mean setting (showing achievability of the universal lower bound).

The rest of this paper is organized as follows. We present our problem setup and universal lower bound in Section~\ref{sec:universallb}, which we then show achievability  in the bounded mean setting in Section~\ref{sec:bounded-mean}. Finally, in Appendix~\ref{sec:complete-proofs} and Appendix~\ref{sec:omittedproofs-two} we provide all the complete proofs.

\section{CADD Lower Bound: Sending $\gamma$ to $\infty$}\label{sec:universallb}

We observe a sequence $X_1,X_2,\dots$ on a filtered probability space $(\Omega,\mathcal F, (\mathcal F_n)_{n\ge1})$ where $\mathcal F_n:=\sigma(X_1, X_2,\dots,X_n)$. Let $\mathcal P$ be a non-empty class of pre-change distributions on the measurable space $(\mathcal X, \mathcal A)$. For $P\in\mathcal P$, write $P^\infty$ for the i.i.d. product measure on $\mathcal X^{\NN}$; in other words, this is the no-change law. Now, let $\mathcal Q$ be a nonempty class of post-change distributions on the measurable space $(\mathcal X, \mathcal A)$. This class may very well be composite/parametric (for example, $\mathcal Q = \{Q_\nu:\nu\in\mathcal V\}$); however, we will work directly at the set-level $\mathcal Q$. For a change time $k\in\{1,2,\dots\}$, a pre-change law $P\in\mathcal P$ and $Q\in\mathcal Q$, denote by
\(
\PP^{(P)}_{k,Q},
\)
the law under which $X_1, \dots, X_{k-1}$ are i.i.d. from $P$, and $X_k,X_{k+1},\dots$ are i.i.d from $Q$. We say that a detection rule is a stopping time $T$ w.r.t. the filtration $(\mathcal F_n)$. Now, we will pose the following remark on the post-change class and information regime. Here, we study asymptotics often as $\gamma\to\infty$ and thereby consider families $T_\gamma$ that satisfy the constraint for each $\gamma$.

Throughout, as explained, $\mathcal Q$ represents the post-change class. Whenever we invoke bounds involving $1/I(Q;\mathcal P)$, we implicitly restrict our attention to those $Q\in\mathcal Q$ for which $0<I(Q;\mathcal P)<\infty$. In addition, for uniform, worst-case over $Q$, statements, we will sometimes assume a positive information gap. Namely, $\underline I=\inf_{Q\in\mathcal Q} I(Q;\mathcal P)>0$. In other words, this intuitively means that the post-change class is uniformly separated from $\mathcal P$ in the $\mathrm{KL}_{\inf}$ sense. Now, in this section, before proceeding, we are first going to present our main theorem, a fully general CADD lower bound. 
\begin{theorem}\label{thm:cadd}
Let $\{T_\gamma\}_{\gamma>0}$ satisfy the ARL constraint $\inf_{P\in\mathcal P}\EE_{P^\infty}[T_\gamma]\ge\gamma$. Then for every post-change law $Q\in\mathcal Q$ with $0<I(Q;\mathcal P)<\infty$,
\[
\boxed{
\quad \liminf_{\gamma\to\infty}\ \frac{\mathcal C_Q(T_\gamma)}{\log\gamma}
\ \ge\ \frac{1}{I(Q;\mathcal P)}.\quad
}
\]
\end{theorem}

\begin{corollary}\label{cor:minimax-cadd-lb}
Let $\{T_\gamma\}_{\gamma>0}$ satisify the ARL constraint, $\inf_{P\in\mathcal P}\EE_{P^\infty}[T_\gamma]\ge \gamma$ for all $\gamma>0$. Assume $\mathcal Q$ contains at least one $Q$ with $0<I(Q;\mathcal P)<\infty$. Then,
\[
\boxed{
\quad \liminf_{\gamma\to\infty}\ \sup_{Q\in\mathcal Q}\ \frac{I(Q;\mathcal P)\,\mathcal C_Q(T_\gamma)}{\log\gamma} \ \ge\ 1.\quad
}
\]
\end{corollary}

\noindent 
Notice that Theorem~\ref{thm:cadd} is a pointwise lower bound in $Q$. On the other hand, Corollary~\ref{cor:minimax-cadd-lb} is a minimax consequence, as it is worse case in $Q$. However, it does not itself imply that $\frac{I(Q;\mathcal P)\mathcal C_Q(T_\gamma)}{\log\gamma}\to 1$ uniformly over all $Q\in\mathcal Q$. As we will see, getting matching uniform upper bounds will necessarily take some additional structure. In particular, we will provide such a uniform achievability result in Section~\ref{sec:bounded-mean} for separated bounded-mean alternatives. More on that later. At this point, we are going to present a series of statements and lemmas that allowed us to develop these bounds. Consider a particular $Q\in\mathcal Q$ with $I(Q,\mathcal P)\in(0,\infty)$.  The first idea we will make clear is specifically for the $\KL$ so that we may work under a particular law. Namely, for every $\delta\in(0,1)$ there exists a $P_{\delta}\in\mathcal P$ such that $0<\KL(Q\|P_\delta)\le I(Q;\mathcal P)+ \delta<\infty.$ In particular, $Q\ll P_\delta$. The reason is because by definition of the infimum, there exists a sequence $(P_n)_{n\ge1}\subset\mathcal P$ with $\KL(Q\|P_n)\downarrow I(Q;\mathcal P)$. Now since $I(Q;\mathcal P)\in(0,\infty)$, we can easily choose $n_\delta$ large enough so that $\KL(Q\|P_{n_\delta}) \le I(Q;\mathcal P)+\delta$. Now, set $P_\delta:=P_{n_\delta}$. Clearly, because $\KL(Q\|P_\delta)<\infty$, necessarily $Q\ll P_\delta$ must follow. Hence $I(Q;\mathcal P)>0$ implies $0<\KL(Q\|P_\delta)$. 

From this point onward, we will fix the pre-change law $\PP^{(P_\delta)}_{k,Q}$ and abbreviate $\PP_{k,Q}:=\PP^{(P_\delta)}_{k,Q}$ and $\PP_\infty:=\PP_{P_\delta^\infty}$. Further, define $\ell_i:=\log\frac{dQ}{dP_\delta}(X_i)$. With the convention that $L_{k:k-1}:=0$, let $L_{k:n}:=\sum_{i=k}^n \ell_i$. Also, let $\mu_\delta:=\KL(Q\|P_\delta)\in(0,\infty)$ and $I_\delta:=I(Q;\mathcal P)+\delta$. It's easy to then see that $\EE_Q[|\ell_1|]<\infty$, and under $\PP_{k,Q}$ the variables $(\ell_i)_{i\ge k}$ are i.i.d.\ with $\EE_{k,Q}[\ell_i]=\mu_\delta$ and $\EE_{k,Q}[|\ell_i|]<\infty$ for $i\ge k$. We can argue why this is the case as follows. First, write $\ell_1=\ell_1^+-\ell_1^-$ with $\ell_1^\pm\ge0$. We know from Lebesgue definition of expectation that $\EE_Q[\ell_1]\in\RR$ if and only if $\EE_Q[\ell_1^+]$ and $\EE_Q[\ell_1^-]$ are finite. Here, $\EE_Q[\ell_1]=\mu_\delta\in(0,\infty)$ so it follows that $\EE_Q[\ell_1^\pm]<\infty$ and $\EE_Q[|\ell_1|]=\EE_Q[\ell_1^+]+\EE_Q[\ell_1^-]<\infty$. Under $\PP_{k,Q}$ the post-change segment $(X_i)_{i\ge k}$ is i.i.d.\ with law $Q$. Hence $(\ell_i)_{i\ge k}$ are i.i.d.\ with the same distribution as $\ell_1$ under $Q$, implying $\EE_{k,Q}[\ell_i]=\mu_\delta$ and $\EE_{k,Q}[|\ell_i|]=\EE_Q[|\ell_1|]<\infty$ for all $i\ge k$. Given all these, we also need to specify measures under the pre-change and post-change laws, which is exactly the motivation for our prefix law equality. To such an end, the prefix and post-change segment are independent, which we also formalize in the following lemma.

\begin{lemma}\label{lem:prefix}
For any $k\ge1$ and any $G\in\mathcal F_{k-1}$, $\PP_{k,Q}(G)=\PP_\infty(G)$. Furthermore, under $\PP_{k,Q}$, the event $\{T\ge k\}\in\mathcal F_{k-1}$ is independent of
$\sigma(X_k,\dots,X_{k+m})$ for every $m\ge0$. In particular, for any event
$E\in\sigma(X_k,\dots,X_{k+m})$ with $\PP_{k,Q}(T\ge k)>0$,
\(
\PP_{k,Q}(E\mid T\ge k)=\PP_{k,Q}(E).
\)
\end{lemma}

\noindent It's often also useful to get an exact expression for the density at an any particular time, which is the motivation for the following lemma, which gives us the exact density at time $n$. In addition, our lemma also presents important change of measure facts.
\begin{lemma}\label{lem:density}
For each $n\ge1$ we have that $\PP_{k,Q}|_{\mathcal F_n}\ll \PP_\infty|_{\mathcal F_n}$. In particular,

\[
\frac{d\PP_{k,Q}|_{\mathcal F_n}}{d\PP_\infty|_{\mathcal F_n}}
=
\begin{cases}
1, & n<k,\\[2pt]
\exp(L_{k,n}), & n\ge k.
\end{cases}
\]
Further, if $A\in\mathcal F_T$, then $\PP_{k,Q}(A) = \EE_\infty\left[e^{L_{k,T}}\1_A\right]$. In addition, take a particular $k\ge1$ with $\PP_\infty(T\ge k)>0$. Then it follows that,
\[
\PP_{k,Q}(A\mid T\ge k)\ =\
\EE_\infty\!\left[e^{L_{k,T}}\1_A\ \middle|\ T\ge k\right].
\]
\end{lemma}

\noindent We also provide a quick lemma for the maximal Strong Law (SLLN) and a corollary that will be useful in our proofs. 
\begin{lemma}\label{lem:maxSLLN}
Let $(Y_i)_{i\ge1}$ be i.i.d.\ with $\EE[|Y_1|]<\infty$ and mean $\mu\ge0$. With $S_m:=\sum_{i=1}^m Y_i$,
\[
\limsup_{n\to\infty}\frac{1}{n}\max_{1\le m\le n}S_m\ = \mu\qquad\text{a.s.}
\]
Consequently, if $b_n$ satisfies $b_n/n\to \mu+\eta$ for some $\eta>0$, then
\[
\PP\!\left(\max_{1\le m\le n}S_m\ >\ b_n\right)\ \longrightarrow\ 0.
\]
\end{lemma}

\noindent We have now all the main ingredients needed for the general CADD lower bound. Let us recap some notation and measurability conditions once again for clarity. For integers $k\le n$ let $L_{k:n}:=\sum_{i=k}^n \ell_i$ with the convention, $L_{k:k-1}=0.$ For $n<k$ we also set $L_{k:n}:=0$ as it is an empty sum. For fixed integers $k$ and $f\in\NN$, we write the event of the alarm occurring in the next $f$ steps after changepoint $k$ as,
\(
A_k:=\{k\le T\le k+f-1\}.
\)
Because $T$ is a stopping time, we know that $\{T=n\}\in\mathcal F_n$ for each $n$, so we have
\(
A_k\cap\{T=n\}=\{k\le n\le k+f-1\}\cap\{T=n\}\in\mathcal F_n,
\)
hence $A_k\in\mathcal F_T$.  Similarly, for any $c\in\RR$,
\(
\{L_{k,T}\le c\}\cap\{T=n\}=\{L_{k:n}\le c\}\cap\{T=n\}\in\mathcal F_n,
\)
so $\{L_{k,T}\le c\}\in\mathcal F_T$.  Finally, $\{T\ge k\}\in\mathcal F_{k-1}$, since $\{T\le k-1\}\in\mathcal F_{k-1}$, and also $\{T\ge k\}\in\mathcal F_T$ because on $\{T=n\}$ it equals $\{n\ge k\}\cap\{T=n\}\in\mathcal F_n$. We will now define a fact that relates the mass of a `block' to the survival at its left endpoint. In other words, we show that the ARL constraint necessitates that there exists at least one block with small conditional null mass. 
\begin{lemma}\label{lem:arl-block}
Let $T$ be any stopping time with $\EE_\infty[T]\ge \gamma$. Fix $f\in\NN$ and partition $\NN$ into disjoint blocks $C_r:=\{(r-1)f+1,\dots,rf\}$, $r\ge1$.
Set $x_r:=\PP_\infty(T\in C_r)$ and $y_r:=\PP_\infty(T\ge (r-1)f+1)$. Then,
\[
\sum_{r\ge1} x_r =1,\qquad \sum_{r\ge1} y_r\ \ge\ \frac{\EE_\infty[T]}{f}\ \ge\ \frac{\gamma}{f}.
\]
Moreover, there exists $r^\star$ with $y_{r^\star}>0$ such that,
\[
\frac{x_{r^\star}}{y_{r^\star}} \le\ \frac{f}{\gamma}.
\]
Equivalently, there exists $k^\star:=(r^\star -1)f+1$ such that for $A_{k^\star}:=\{k^\star\le T\le k^\star+f-1\}$, 
\[
\PP_\infty(A_{k^\star}\mid T\ge k^\star)\ \le\ \frac{f}{\gamma}.
\]
\end{lemma}

\noindent Finally, we present a lemma that gives intuition on the asymptotics of parameters.
\begin{lemma}\label{lem:par-asymp}
Take a particular but arbitrary $\varepsilon,\delta\in(0,1)$; set $I_\delta:=I(Q;\mathcal P)+\delta$ and 
$\mu_\delta:=\KL(Q\|P_\delta)\le I_\delta$. Then define,
\[
f_\gamma = \Big\lfloor \frac{(1-\varepsilon)\log\gamma}{I_\delta}\Big\rfloor,\qquad
c_\gamma:=b\log\gamma\quad\text{for any }b\text{ with }
\frac{(1-\varepsilon)\mu_\delta}{I_\delta}\ <\ b\ <\ 1.
\]
Then, as $\gamma\to\infty$, 
\[
\frac{f_\gamma}{\log\gamma}\to \frac{1-\varepsilon}{I_\delta},\qquad
\frac{c_\gamma}{f_\gamma}\to \frac{b}{1-\varepsilon}\,I_\delta>\mu_\delta,\qquad
\frac{e^{c_\gamma}f_\gamma}{\gamma}\to 0.
\]
In particular, there exist $\eta>0$ and $\gamma_0$ such that $\frac{c_\gamma}{f_\gamma}\ge \mu_\delta+\eta$  for all $\gamma\ge\gamma_0$.
\end{lemma}

\noindent We are now ready to prove Theorem~\ref{thm:cadd}.
\begin{proof}[Proof of Theorem~\ref{thm:cadd}]
Take an arbitrary $Q\in\mathcal Q$ with $0<I(Q;\mathcal P)<\infty$. Our following argument is in fact pointwise in $Q$ and hence applies to all such $Q\in\mathcal Q$. Our proof will proceed as follows. Take an $\varepsilon\in(0,1)$ and $\delta\in(0,1)$; take $P_\delta$ as the near minimizer, and abbreviate $I_\delta:=I(Q;\mathcal P)+\delta$ and
$\mu_\delta:=\KL(Q\|P_\delta)\le I_\delta$. Define
\[
f_\gamma:=\Big\lfloor \frac{(1-\varepsilon)\log\gamma}{I_\delta}\Big\rfloor,
\qquad
c_\gamma:=b\log\gamma\quad\text{for any }b\text{ with }
\frac{(1-\varepsilon)\mu_\delta}{I_\delta}\ <\ b\ <\ 1.
\]
Then for each $\gamma$, we will show that there exists an index $k_\gamma$ such that
\begin{equation}\label{eq:eps-delta}
\EE_{k_\gamma,Q}\!\left[(T_\gamma-k_\gamma+1)^+\ \middle|\ T_\gamma\ge k_\gamma\right]
\ \ge\ f_\gamma\,(1-o(1)).
\end{equation}
(All conditional probabilities that will follow in this proof are well-defined because we will show that \(\PP_\infty(T_\gamma\ge k_\gamma)>0\), and then obviously \(\PP_{k_\gamma,Q}(T_\gamma\ge k_\gamma)=\PP_\infty(T_\gamma\ge k_\gamma)>0)\) by Lemma~\ref{lem:prefix}.)
Consequently, this implies that
\[
\liminf_{\gamma\to\infty}\ \frac{\mathcal C_Q(T_\gamma)}{\log\gamma}
\ \ge\ \frac{1-\varepsilon}{I_\delta}.
\]
Now, letting $\delta\downarrow0$ and then $\varepsilon\downarrow0$ yields the theorem. So now it remains to prove~\eqref{eq:eps-delta}. Firstly, from Lemma~\ref{lem:par-asymp}, there must exist an $\eta>0$ and $\gamma_0$ whereby $\frac{c_\gamma}{f_\gamma}\ge \mu_\delta+\eta$ for all $\gamma\ge\gamma_0$. Next, we will choose a block with small conditional null mass. Intuitively, we divide time into blocks of length $f_\gamma$, recognizing that under the pre-change distribution, Lemma~\ref{lem:arl-block} implies that there is a (special) block has a small chance of false-alarming inside of it given its survival to its start. This is the window we will analyze. 

To elaborate, apply Lemma~\ref{lem:arl-block} to $T_\gamma$ with $f=f_\gamma$ and let $k_\gamma=(r^\star_\gamma-1)f_\gamma+1$ be the left endpoint it returns, for which $y_{r^\star_\gamma}=\PP_\infty(T_\gamma\ge k_\gamma)>0$ and,
\begin{equation}\label{eq:null-cond}
\PP_\infty(A_{k_\gamma}\mid T_\gamma\ge k_\gamma)\ \le\ \frac{f_\gamma}{\gamma},
\qquad A_{k}:=\{k\le T_\gamma\le k+f_\gamma-1\}.
\end{equation}
Now, we will apply the conditional-change of measure to split the event of an alarm with low cumulative evidence, and an alarm with high cumulative evidence. We will show that both of these cases are rare. Namely, we apply Lemma~\ref{lem:density} with $A=A_k$ and conditioning on $\{T_\gamma\ge k\}$,
\[
\PP_{k,Q}(A_k\mid T_\gamma\ge k)
=\EE_\infty\!\left[e^{L_{k,T_\gamma}}\1_{A_k}\,\middle|\,T_\gamma\ge k\right].
\]
We can now quickly split on $\{L_{k,T_\gamma}\le c_\gamma\}$, which obviously belongs to $\mathcal F_{T_\gamma}$,
\begin{align*}
\PP_{k,Q}(A_k\mid T_\gamma\ge k)
&= \EE_\infty\!\left[e^{L_{k,T_\gamma}}\1_{A_k\cap\{L_{k,T_\gamma}\le c_\gamma\}}\ \middle|\ T_\gamma\ge k\right]\\
&\quad + \EE_\infty\!\left[e^{L_{k,T_\gamma}}\1_{A_k\cap\{L_{k,T_\gamma}>c_\gamma\}}\ \middle|\ T_\gamma\ge k\right]\\
&\le e^{c_\gamma}\PP_\infty(A_k\mid T_\gamma\ge k)
   + \PP_{k,Q}\!\left(A_k\cap\{L_{k,T_\gamma}>c_\gamma\}\ \middle|\ T_\gamma\ge k\right)\\
&=: (\mathrm I)+(\mathrm{II}),
\end{align*}
where for the second term we used Lemma~\ref{lem:density} once again. We will now quickly bound $(\mathrm I)$ via the ARL constraint. By  \eqref{eq:null-cond} and Lemma~\ref{lem:par-asymp}, $(\mathrm I)\le e^{c_\gamma}\,f_\gamma/\gamma
= f_\gamma\,\gamma^{b-1}\xrightarrow[\gamma\to\infty]{}0$, which is immediate from the fact that $b<1$ and $f_\gamma=O(\log\gamma)$ while $\log\gamma=o(\gamma^\alpha)$ for any $\alpha>0$. Proceeding by bounding $(\mathrm{II})$ via maximal SLLN, conditionally, we know that $T_\gamma\le k+f_\gamma-1$ so it follows that,
\[
A_k\cap\{L_{k,T_\gamma}>c_\gamma\}\ \subseteq\
\Big\{\max_{1\le m\le f_\gamma}L_{k:k+m-1}>c_\gamma\Big\}.
\]
Here, the right-hand event belongs to $\sigma(X_k,\dots,X_{k+f_\gamma-1})$ whence by Lemma~\ref{lem:prefix} and the definition of $(\mathrm{II})$,
\[
(\mathrm{II})
\le \PP_{k,Q}\!\left(\max_{1\le m\le f_\gamma}L_{k:k+m-1}>c_\gamma\ \middle|\ T_\gamma\ge k\right)
= \PP_{k,Q}\!\left(\max_{1\le m\le f_\gamma}L_{k:k+m-1}>c_\gamma\right).
\]
Under $\PP_{k,Q}$, we know that the the increments $(\ell_i)_{i\ge k}$ are i.i.d.\ with mean $\mu_\delta$ and
$\EE_{k,Q}[|\ell_i|]<\infty$. And, by Lemma~\ref{lem:par-asymp}, it follows that,
$\frac{c_\gamma}{f_\gamma}\ge \mu_\delta+\eta$ for all $\gamma\ge\gamma_0$; therefore, the corollary in Lemma~\ref{lem:maxSLLN} yields $(\mathrm{II})\to0$. If we now combine these last two results, this gives us
\[
\PP_{k_\gamma,Q}\!\left(T_\gamma\le k_\gamma+f_\gamma-1\ \middle|\ T_\gamma\ge k_\gamma\right)\ \xrightarrow{}\ 0,
\]
and hence
\[
\PP_{k_\gamma,Q}\!\left(T_\gamma\ge k_\gamma+f_\gamma\ \middle|\ T_\gamma\ge k_\gamma\right)\ \xrightarrow{}\ 1.
\]
By conditional Markov it follows that,
\[
\EE_{k_\gamma,Q}\!\left[(T_\gamma-k_\gamma+1)^+\ \middle|\ T_\gamma\ge k_\gamma\right]
\ \ge\ f_\gamma\ \PP_{k_\gamma,Q}\!\left(T_\gamma\ge k_\gamma+f_\gamma\ \middle|\ T_\gamma\ge k_\gamma\right)
\ =\ f_\gamma\,(1-o(1)).
\]

\noindent Taking the supremum over $k$ yields $\mathcal C_Q(T_\gamma)\ge f_\gamma(1-o(1))$. Finally, Lemma~\ref{lem:par-asymp} gives
$f_\gamma/\log\gamma=\frac{1-\varepsilon}{I_\delta}+o(1)$, proving~\eqref{eq:eps-delta},
and letting $\delta\downarrow0$ then $\varepsilon\downarrow0$ completes the proof. Since our above argument holds for an arbitrary $Q\in\mathcal Q$ that satisfies $0<I(Q;\mathcal P)<\infty$, the bound must in fact hold for all $Q\in\mathcal Q$.
\end{proof}

\noindent Note that in the following section, all of the bounds on $\mathcal C_Q(T^{\mathrm{BM}}_\gamma)$ below hold under each and every pre-change law $P\in\mathcal P_m$. Because, the post-change segment is iid $Q$ and independent of $\mathcal F_{k-1}$ under $\PP^{(P)}_{k,Q}$. Meaning, the resulting upper bound we will present below applies to any least-favorable pre-change distribution used in the above general lower bound.

\section{Asymptotically Optimal E-detector for the Bounded Mean Problem}\label{sec:bounded-mean}

In this section, let us have some fun by specializing the general lower bound in Theorem~\ref{thm:cadd} to the classical bounded mean setting. In doing so, we will show that the lower bound is tight in the ``pure ARL$\to\infty$'' regime. That is, there exists a family of stopping rules $\{T_\gamma\}_{\gamma>0}$ that satisfy the ARL constraint and achieve the matching $\log\gamma$-delay constant for every fixed post-change law $Q$. Further, throughout this section we will take $\mathcal X=[0,1]$ and consider a known baseline $m\in(0,1)$. Let us denote the pre-change class $\mathcal P_m:=\Big\{\mathcal P[0,1]: \EE_P[X_1]\le m\Big\}$. And, let us denote the post-change class as $\mathcal Q_m:=\Big\{\mathcal Q[0,1]: \EE_Q[X_1]>m\Big\}$, which clearly corresponds to an increase above the baseline mean. For such a $Q\in\mathcal Q_m$, define the information quantity as $\mathrm{KL}_{\inf}(Q;m) := I(Q;\mathcal P_m) = \inf_{P\in\mathcal P_m}\KL(Q\|P)$. 

To recap, from our Theorem~\ref{thm:cadd} any family of stopping rules $\{T_\gamma\}$ which satisfies the ARL constraint $\inf_{P\in\mathcal P_m}\EE_{P^\infty}[T_\gamma]\ge \gamma$ must indeed obey the universal lower bound on detection delay of $\log \gamma/\mathrm{KL}_{\inf}(Q;m)$ (as $\gamma \to \infty$).
 In what follows, we're going to show that this bound is achievable, therefore tight. As such, we will establish asymptotic optimality (also for the uniform notion in Corollary~\ref{cor:minimax-cadd-lb}). In addition, we will construct a detector that attains a matching uniform upper bound over separated subclasses of $\mathcal Q_m$. Let us first define something we need for building a variational representation of $\mathrm{KL}_{\inf}(Q;m)$. Meaning, for $\lambda\in(0,1)$ and $x\in[0,1]$, let us define
\begin{equation}\label{eq:bm-Llambda}
L_\lambda(x):=1+\lambda\Big(\frac{x}{m}-1\Big)=(1-\lambda)+\frac{\lambda}{m}x,
\end{equation}
which is an e-value for $\mathcal P_m$ (an e-value is a nonnegative random variable whose expectation is at most one under every $P \in \mathcal P_m$).
\noindent Let us now define our mixture Shiryaev-Roberts (SR) type detector with ARL control. It is in fact exactly an e-detector as in \citet{shin2024edetectors}. ARL control when thresholding an e-detector is immediate from their results, but we still provide a proof for it in our theorem below for completeness. Now, consider a countable dense set\footnote{For example, the dyadic rationals.} $\Lambda\subset(0,1)$, and pick weights $(w_\lambda)_{\lambda\in\Lambda}$ with $w_\lambda>0$ and $\sum_{\lambda\in\Lambda}w_\lambda=1$. Then for each $\lambda\in\Lambda$, define the SR-statistic as,
\[
R^{(\lambda)}_n:=\sum_{k=1}^n \prod_{i=k}^n L_\lambda(X_i), \quad \text{where } R^{(\lambda)}_0:=0,
\]
and $L_\lambda$ is the same as what is defined in \eqref{eq:bm-Llambda}. Obviously, this statistic has the following the one-step recursion which we can see from factoring out $L_\lambda(X_n)$,
\begin{equation}\label{eq:bm-SR-rec}
R^{(\lambda)}_n =\big(1+R^{(\lambda)}_{n-1}\big)L_\lambda(X_n).  
\end{equation}

\noindent Now, let us form the mixture statistic $M_n:=\sum_{\lambda\in\Lambda}w_\lambda R^{(\lambda)}_n$. And then, let us define the following stopping rule, which will be our bounded-mean detector, as,
\begin{equation}\label{eq:bm-Tgamma}
T^{\mathrm{BM}}_\gamma := \inf\{n\ge1: M_n\ge \gamma\}.    
\end{equation}

\noindent Given all this, we now present our achievability theorem below.
\begin{theorem}\label{thm:bm-opt}
Take a particular $m\in(0,1)$ and define $\mathcal P_m$ and let $T^{\mathrm{BM}}_\gamma$ be as in \eqref{eq:bm-Tgamma}. Then:
\begin{enumerate}
    \item Firstly, we have ARL-calibration: $\inf_{P\in\mathcal P_m}\EE_{P^\infty}[T^\mathrm{BM}_\gamma]\ge \gamma$ for all $\gamma>0$.
    \item Secondly, we in fact have asymptotic optimality in this regime. Meaning, for every $Q\in\mathcal Q_m$,
    \[
    \boxed{
    \quad \lim_{\gamma\to\infty}\frac{\mathcal C_Q(T^{\mathrm{BM}}_\gamma)}{\log\gamma} \ =\ \frac{1}{\mathrm{KL}_{\inf}(Q;m)}.\quad
    }
    \]
    \item Third, we have uniform minimax optimality on separated classes: for any $\Delta\in(0,1-m)$, 
    \[
    \boxed{
    \quad \lim_{\gamma\to\infty}\ \sup_{Q\in\mathcal Q_{m+\Delta}}\ \frac{\mathrm{KL}_{\inf}(Q;m)\,\mathcal C_Q(T^{\mathrm{BM}}_\gamma)}{\log\gamma}\ =\ 1.\quad
    }
    \]
\end{enumerate}
\end{theorem}

\noindent We describe the key items we need for this theorem as follows. An immediate one is that we need to verify both positivity and finiteness of $\mathrm{KL}_{\inf}(Q;m)$. In addition, to get uniformity, we need a uniform information gap under mean separation. Finally, we use a well-known fact from the works of \citet{HondaTakemura2010} which tell us that the entire minimization over all distributions $P$ with the mean constraint becomes a simple one-dimensional maximization over $\lambda$. We formalize all that with the following lemma. 

\begin{lemma}\label{lem:bm-pos-fin}
Suppose that $Q$ is supported on $[0,1]$ and $\EE_Q[X_1]=q>m$. Then we have that $0<\mathrm{KL}_{\inf}(Q;m)<\infty$.  In addition, $\mathrm{KL}_{\inf}(Q;m) = \sup_{\lambda\in[0,1]}\EE_Q[\log L_\lambda(X_1)]$. Finally, for any arbitrary $\Delta\in(0,1-m)$, we have that $\inf_{Q\in\mathcal Q_{m+\Delta}} \mathrm{KL}_{\inf}(Q;m)\ge 2\Delta^2$.
\end{lemma}

\noindent With all of this that we have now developed, for the actual asymptotic optimality results in our theorem, we will only need the following consequence of the second statement in the above Lemma. The fact that for every $\eta>0$ there exists some $\lambda\in(0,1)$ with $\EE_Q[\log L_\lambda(X_1)]\ge \mathrm{KL}_{\inf}(Q;m)-\eta$. In other words, we do not need any closed form optimizer $\lambda^\star$. Now, before we proceed to our main theorem's proof, we will need one more hitting-time fact for positive-drift random walks: in other words, for the bounded increments we will derive some hitting time asymptotics.

\begin{lemma}\label{lem:bm-hit}
Let $(Y_i)_{i\ge1}$ be iid with $\EE[Y_1]=\mu>0$ and assume that $Y_1\in[a,b]$ almost surely for some finite $a<b$. Define $S_n:=\sum_{i=1}^n Y_i$ and for $u>0$, $\tau_u:=\inf\{n\ge 1: S_n\ge u\}$. Then, $\tau_u<\infty$ almost surely and $\EE[\tau_u]<\infty$ for every $u>0$. Moreover for every $u>0$,
\begin{equation}\label{eq:bm-hit-bounds}
\frac{u}{\mu}\le \EE[\tau_u] \le \frac{u+b}{\mu}.   
\end{equation}
In addition,
\[
\lim_{u\to\infty}\frac{\EE[\tau_u]}{u} = \frac{1}{\mu}.
\]
\end{lemma}

\noindent We now present the proof for the asymptotic optimality statement in our Theorem~\ref{thm:bm-opt}. We defer both the ARL calibration (the first statement of the theorem) and the uniform minimax proof (the third statement of the theorem) to the appendix.

\smallskip
\begin{proof}[Proof of the Asymptotic Optimality Claim in Theorem~\ref{thm:bm-opt}.]
We will now prove the second statement. We will start this proof by first reducing the detector if you will, to a single ``good'' $\lambda$. As such, to begin let us consider some arbitrary $\eta\in(0,\mathrm{KL}_{\inf}(Q;m))$. We will show that 
\[
\limsup_{\gamma\to\infty}\frac{\mathcal C_Q(T^{\mathrm{BM}}_\gamma)}{\log\gamma} \le \frac{1}{\mathrm{KL}_{\inf}(Q;m)-\eta},
\]
which yields the theorem by letting $\eta \downarrow 0$.

For some $\lambda\in[0,1)$, let's define $g(\lambda):=\EE_Q[\log L_\lambda(X_1)]$ so that $\sup_{\lambda\in[0,1]}g(\lambda)=\mathrm{KL}_{\inf}(Q;m)$ by Lemma~\ref{lem:bm-pos-fin}. Furthermore because of the fact that $\lambda\mapsto \log L_\lambda(x)$ is concave in $\lambda$ for each $x$, it follows that the function $g(\lambda)=\EE_Q[\log L_\lambda(X_1)]$ is concave on $(0,1)$. Therefore, it must admit a left limit at $1$, which could possibly be $-\infty$, so $\sup_{\lambda\in[0,1]}g(\lambda)=\sup_{\lambda\in(0,1)}g(\lambda)$. Now, we know that $L_\lambda$ is continuous in $\lambda$ for each $x$ and is bounded away from $0$ uniformly over $\lambda\in[0,1-\delta]$ for any $\delta>0$. As a result, the map $\lambda\mapsto g(\lambda)$ indeed is continuous on $[0,1)$, and could possibly take value $-\infty$ at $\lambda=1$ should $Q(\{0\})>0 $. Furthermore, we know that because $\Lambda$ is dense in $(0,1)$, we can easily choose $\lambda_\eta\in\Lambda$ such that,
\[
d_\eta := g(\lambda_\eta) = \EE_Q[\log L_{\lambda_\eta}(X_1)] \ge \mathrm{KL}_{\inf}(Q;m)-\eta > 0.
\]

\noindent Let $w_\eta:=w_{\lambda_\eta}>0$ be its mixture weight. Note that necessarily $M_n\ge w_\eta R^{(\lambda_\eta)}_n$ for all $n$. As a consequence we of course get that,
\[
\{R^{(\lambda_\eta)}_n \ge \gamma/w_\eta\} \subseteq \{M_n\ge\gamma\}.
\]

\noindent What this entails is that the stopping time $T^{\mathrm{BM}}_\gamma$ is always no larger than the stopping time that we would get from using only the single component $\lambda_\eta$ with the threshold $\gamma/w_\eta$. To formalize that intuition, let us define $T^{(\lambda_\eta)}_{\gamma/w_\eta}:=\inf\{n\ge 1: R^{(\lambda_\eta)}_n\ge \gamma/w_\eta\}$. Then as a consequence we have that with probability one,
\begin{equation}\label{eq:bm-T-dom}
 T^{\mathrm{BM}}_\gamma \le T^{(\lambda_\eta)}_{\gamma/w_\eta}.   
\end{equation}

\noindent So, we have formalized the idea that the mixture can only stop earlier than any single component that has been scaled by its weight. With this being said, we will now argue that after the changepoint, one SR term is enough for us, ie that $R^{(\lambda_\eta)}_n$ contains the correct product. As such, consider a particular changepoint $k\ge 1$. For any pre-change law $P\in\mathcal P_m$, under $\PP^{(P)}_{k,Q}$ we know that by definition the post-change segment $(X_k, X_{k+1},\dots)$ is iid $Q$ and is independent of $\mathcal F_{k-1}$. That is, the event $\{T^{\mathrm{BM}}_\gamma\ge k\}\in\mathcal F_{k-1}$ is indeed independent of the post-change observations as it is contained in the generated sigma algebra of the pre-change observations. Now, note that for $n\ge k$, $R^{(\lambda_\eta)}_n$ contains the $j=k$ summand; that is,
\[
R^{(\lambda_\eta)}_n = \sum_{j=1}^n \prod_{i=j}^n L_{\lambda_\eta}(X_i) \ge \prod_{i=k}^n L_{\lambda_\eta}(X_i).
\]

\noindent Intuitively, what we are illustrating above is that if indeed the product from $k$ onwards hits $\gamma/w_\eta$, $R^{(\lambda_\eta)}_n$ must also. Now, taking logs, let us let $S_t:=\sum_{i=1}^t Y_i$ and $Y_i:=\log L_{\lambda_\eta}(X_{k+i-1})$ for $t\ge 1$. Then it follows that for $t\ge 1$,
\[
\prod_{i=k}^{k+t-1}L_{\lambda_\eta}(X_i)=\exp(S_t).
\]

\noindent Now define the post-change crossing time as $\tau^{(\eta)}_{u}:=\inf\{t\ge 1: S_t\ge u\}$. With $u_\gamma:=\log(\gamma/w_\eta)$, we can conclude that $S_{\tau^{(\eta)}_{u_\gamma}} \ge u_\gamma$. Therefore,
\[
\prod_{i=k}^{k+\tau^{(\eta)}_{u\gamma}-1}L_{\lambda_\eta}(X_i) \ge \gamma/w_\eta \quad \Longrightarrow R^{(\lambda_\eta)}_{k+\tau^{(\eta)}_{u_\gamma}-1} \ge \gamma/w_\eta.
\]

\noindent Hence, $T^{(\lambda_\eta)}_{\gamma/w_\eta}\le k+\tau^{(\eta)}_{u_\gamma}-1$. So it must be the case that $(T^{(\lambda_\eta)}_{\gamma/w_\eta}-k+1)^+ \le \tau^{(\eta)}_{u_\gamma}$. Now if we combine this result with \eqref{eq:bm-T-dom}, we get the pathwise inequality which tells us that the detection delay is bounded by this hitting time. Meaning, under $\PP^{(P)}_{k,Q}$ for every $k\ge1$,
\begin{equation}\label{eq:bm-delay-dom}
(T^{(\lambda_\eta)}_{\gamma/w_\eta}-k+1)^+ \le \tau^{(\eta)}_{u_\gamma}. 
\end{equation}

\noindent Let us now argue why we can drop the conditioning on surivial, ie conditioning on $\{T^{\mathrm{BM}}_\gamma\ge k\}$. We know that $\{T^{\mathrm{BM}}_\gamma \ge k\}\in\mathcal F_{k-1}$ and $\tau^{(\eta)}_{u_\gamma}$ depends only on $(X_k,X_{k+1}, \dots)$, and hence under $\PP^{(P)}_{k,Q}$ the two are independent. So conditioning on $\{T^{\mathrm{BM}}_\gamma \ge k\}$ does not change the distribution of $\tau^{(\eta)}_{u_\gamma}$. Therefore,
\[
\EE^{(P)}_{k,Q}[\tau^{(\eta)}_{u_\gamma}\, | \, T^{\mathrm{BM}}_\gamma \ge k] = \EE_Q[\tau^{(\eta)}_{u_\gamma}].
\]

\noindent In addition, from \eqref{eq:bm-delay-dom} and using monotonicity of conditional expectation, it is easy to see that,
\[
\EE^{(P)}_{k,Q}[(T^{\mathrm{BM}}_\gamma - k+1)^+\,| \, T^{\mathrm{BM}}_\gamma\ge k] \le \EE_Q[\tau^{(\eta)}_{u_\gamma}].
\]

\noindent Now all we need to do is take the supremum over $k\ge 1$ and we obtain that,
\begin{equation}\label{eq:bm-CADD-upper}
 \mathcal C_Q(T^{\mathrm{BM}}_\gamma) \le \EE_Q[\tau^{(\eta)}_{u_\gamma}].   
\end{equation}

\noindent We will conclude the proof by deriving asymptotics on the crossing time. Under the post-change law $Q$, we know that the increments of $Y_i=\log L_{\lambda_\eta}(X_i)$ are iid and bounded. Because, $X_i\in[0,1]$ implies that $1-\lambda_\eta \le L_{\lambda_\eta}(X_i) \le (1-\lambda_\eta)+\frac{\lambda_\eta}{m}$, so $\log L_{\lambda_\eta}(X_i)$ is bounded indeed. In addition, we know that $\EE_Q[Y_1]=d_\eta>0$ by construction. Therefore we can apply Lemma~\ref{lem:bm-hit} with $\mu=d_\eta$ to give us as $\gamma\to\infty$ that,
\[
\frac{\EE_Q[\tau^{(\eta)}_{u_\gamma}]}{u_\gamma} \longrightarrow \frac{1}{d_\eta}.
\]

\noindent We know that $u_\gamma=\log(\gamma/w_\eta)=\log\gamma + O(1)$, and so because of this, we get that $u_\gamma/\log\gamma\to1$. And thus,
\[
\limsup_{\gamma\to\infty}\frac{\EE_Q[\tau^{(\eta)}_{u_\gamma}]}{\log\gamma} = \frac{1}{d_\eta}.
\]

\noindent Lastly, if we leverage \eqref{eq:bm-CADD-upper} alongside this, we get that,
\[
\limsup_{\gamma\to\infty}\frac{\mathcal C_Q(T^{\mathrm{BM}}_\gamma)}{\log\gamma} \le \frac{1}{d_\eta} \le \frac{1}{\mathrm{KL}_{\inf}(Q;m)-\eta}.
\]

\noindent Now all we need to do is let $\eta\downarrow 0$ and we get a sharp upper bound. Further, the matching lower bound is an immediate specialization of Theorem~\ref{thm:cadd} with $\mathcal P=\mathcal P_m$. So if we combine the lower and upper bounds we get the tight limit, which is exactly the second statement of our theorem. 
\end{proof}

\noindent One very obvious thing we will point out is how to extend the results of this section to general bounded intervals. Meaning, if $X_i\in[a,b]$ almost surely, of course we can reduce to the $[0,1]$ case via the transformation $\widetilde X_i:= (X_i-a)/(b-a)\in[0,1]$ and correspondingly $\widetilde m:=(m-a)/(b-a)$. To that end, the resulting procedure and constant $\mathrm{KL}_{\inf}(Q;m)$ are going to transform accordingly.
\section{Conclusion}
In composite changepoint detection, we are forced to deal with the fact that the pre-change distribution that is hardest to distinguish from a particular $Q$ is not known in advance. As such, the right information quantity is indeed the $\KLinf$ projection $I(Q;\mathcal P)$. In this paper, we prove that this particular quantity drives the first order CADD asymptotics under an ARL constraint via our universal lower bound \eqref{eq:mainlb}. We then showed tightness in the bounded-mean model by constructing a mixture SR stopping rule based on one step betting factors $L_\lambda$ (whose log growth matches $\KLinf(Q;m)$). Of course, there are several avenues of future directions of work. Firstly, an immediate one is extending beyond bounded mean detection using this same framework. Second, it would be interesting to interpolate between our i.i.d.\ sharp constant analysis to the nonasymptotic non-i.i.d.\ e-detector framework \citep{shin2024edetectors, RufLarssonKoolenRamdas2023} in some way. Doing this may give even better stopping rules that are not only distribution free under such large composite nulls but computationally efficient and constant optimal in $\log\gamma$ asymptotics (or other such regimes) when additional structure (eg, independence) is indeed there.

\bibliography{references}
\appendix
\section{Omitted Proofs for Section~\ref{sec:universallb}}\label{sec:complete-proofs}

\begin{proof}[Proof of Lemma~\ref{lem:prefix}]
Let $\mathscr{C}$ be the collection of cylinder sets in $\mathcal F_{k-1}$ of the form $C=\{(x_1,\dots,x_{k-1})\in A\}$ with $A\in\mathcal A^{\otimes(k-1)}$. Under $\PP_{k,Q}$ and $\PP_\infty$, the finite-dimensional distributions of $(X_1,\dots,X_{k-1})$ coincide and equal $P_\delta^{k-1}$, so $\PP_{k,Q}(C)=\PP_\infty(C)$ for all $C\in\mathscr{C}$. Since $\mathscr{C}$ is clearly a $\pi$-system generating $\mathcal F_{k-1}$ and both measures agree on it, by the Dynkin $\pi$-$\lambda$ theorem they must agree on all of $\mathcal F_{k-1}$.

It remains to prove the second claim of our lemma. Under $\PP_{k,Q}$ the vector $(X_1,\dots,X_{k-1})$ has law $P_\delta^{k-1}$ and
is independent of $(X_k,X_{k+1},\dots)$, which is i.i.d.\ $Q$. Since $T$ is a stopping time,
$\{T\ge k\}\in\mathcal F_{k-1}$ and is thus a function of $(X_1,\dots,X_{k-1})$ only.
Therefore $\{T\ge k\}$ is independent of $\sigma(X_k,\dots)$, proving the claim.
\end{proof}

\begin{proof}[Proof of Lemma~\ref{lem:density}]
We start with the \underline{first claim} of the lemma. For $n<k$ we have $(X_1,\dots,X_n)\sim P_\delta^n$ under both measures, hence the density equals $1$. For $n\ge k$, we use independence and the chain rule for Radon-Nikodym derivatives. For $A\in\mathcal F_n$,
\[
\PP_{k,Q}(A)
=\int \left(\prod_{i=1}^{k-1} dP_\delta(x_i)\right)\left(\prod_{i=k}^{n} dQ(x_i)\right)\1_A
=\int \left(\prod_{i=1}^{n} dP_\delta(x_i)\right)\,
\prod_{i=k}^n\frac{dQ}{dP_\delta}(x_i)\,\1_A.
\]
Thus the density with respect to $P_\delta^n$ (namely, $\PP_\infty|_{\mathcal F_n}$) equals
$\prod_{i=k}^n \frac{dQ}{dP_\delta}(X_i)=\exp(L_{k:n})$. Since $Q\ll P_\delta$, absolute continuity must follow therefore. \underline{We now prove the second claim of the lemma.} Because $T$ is a stopping time, $\{T=n\}\in\mathcal F_n$ and $A\cap\{T=n\}\in\mathcal F_n$ for all $n$. Also, on the event $\{T=n\}$ we have that $L_{k,T}=L_{k:n}$, and also that $e^{L_{k,T}}\1_{A\cap\{T=n\}}$ is non-negative and $\mathcal F_n$-measurable. By decomposition then,
\[
1_A=\sum_{n\ge1}\1_{A\cap\{T=n\}}\quad\text{(disjoint union).}
\]

\noindent It follows then by the first claim of Lemma~\ref{lem:density} and the monotone convergence theorem that,
\begin{align*}
\PP_{k,Q}(A)
&=\sum_{n\ge1}\PP_{k,Q}(A\cap\{T=n\})
=\sum_{n\ge1}\EE_\infty\!\left[\frac{d\PP_{k,Q}|_{\mathcal F_n}}{d\PP_\infty|_{\mathcal F_n}}
\1_{A\cap\{T=n\}}\right]\\
&=\sum_{n\ge1}\EE_\infty\!\left[e^{L_{k:n}}\1_{A\cap\{T=n\}}\right]
=\sum_{n\ge1}\EE_\infty\!\left[e^{L_{k:T}}\1_{A\cap\{T=n\}}\right]
=\EE_\infty\!\left[e^{L_{k:T}}\sum_{n\ge1}\1_{A\cap\{T=n\}}\right]\\
&=\EE_\infty\!\left[e^{L_{k:T}}\1_A\right].
\end{align*}

\noindent We now finish our proof by proving the \underline{third claim} of the lemma. By the second claim of Lemma~\ref{lem:density}, for any $B\in\mathcal F_T$, we have that,
$\PP_{k,Q}(B)=\EE_\infty[e^{L_{k,T}}\1_B]$. Applying this to $B=A\cap\{T\ge k\}$, and then dividing by $\PP_{k,Q}(T\ge k)=\PP_\infty(T\ge k)>0$ (the equality of these two probabilities follows from Lemma~\ref{lem:prefix}) yields the conditional identity and we are done.
\end{proof}

\begin{proof}[Proof of Lemma~\ref{lem:maxSLLN}]
We know that by the strong law of large numbers and the fact that $\EE[|Y_1|]<\infty$, there exists an almost sure event $\Omega_0$ on which,
\[
\lim_{m\to\infty}\frac{S_m(\omega)}{m}=\mu.
\]
Now, fix some $\omega\in\Omega_0$ and $\varepsilon\in(0,1)$, Then, there exists $N_\varepsilon(\omega)$ so that $S_m(\omega)\le (\mu+\varepsilon)m$ for all $m\ge N_\varepsilon(\omega)$. For this $n\ge N_\varepsilon(\omega)$,
\[
\max_{1\le m\le n}S_m(\omega)
=\max\Big\{\max_{1\le m< N_\varepsilon(\omega)}S_m(\omega),\
\max_{N_\varepsilon(\omega)\le m\le n}S_m(\omega)\Big\}
\le C_\varepsilon(\omega)+(\mu+\varepsilon)n,
\]
where the finite sum $C_\varepsilon(\omega):=\max_{1\le m< N_\varepsilon(\omega)}S_m(\omega)<\infty$.  Hence dividing by $n$ and taking the $\limsup$ gives us,
\[
\limsup_{n\to\infty}\frac{1}{n}\max_{1\le m\le n}S_m(\omega)\
\le\ \mu+\varepsilon.
\]
Since $\varepsilon>0$ is arbitrary, the first claim must hold almost surely. To show the probability bound, what we need to do is first take an $\eta>0$ and choose $\varepsilon\in(0,\eta)$. Now, on that same almost sure event, for all $n$ large enough,
\[
\max_{1\le m\le n}S_m \ \le\ C_\varepsilon+(\mu+\varepsilon)n < (\mu+\eta)n\ \le\ b_n,
\]
since $b_n/n\to \mu+\eta$.  This means that the events $\{\max_{1\le m\le n}S_m>b_n\}$ must occur only finitely often almost surely. Hence, their probabilities tend to 0.
\end{proof}

\begin{proof}[Proof of Lemma~\ref{lem:arl-block}]
Let $s_n:=\PP_\infty(T\ge n)$, ie, nonincreasing. We know by definition of expectations that $\EE_\infty[T]=\sum_{n\ge1}s_n$ holds  in $[0,\infty]$. So in each block,
$C_r:=\{(r-1)f+1,\dots,rf\}$ we have that,
\[
s_{(r-1)f+1}\ \ge\ \frac1f\sum_{n\in C_r} s_n,
\]
hence,
\[
\sum_{r\ge1} y_r=\sum_{r\ge1} s_{(r-1)f+1}\ \ge\ \frac1f\sum_{r\ge1}\sum_{n\in C_r} s_n
=\frac{1}{f}\sum_{n\ge1}s_n=\frac{\EE_\infty[T]}{f}\ \ge\ \frac{\gamma}{f}.
\]
Also $\sum_{r\ge1}x_r=1$ by disjointness; define $r_r:=x_r/y_r$ for $y_r>0$ and set $r_r:=0$ when $y_r=0$. Note that $x_r\le  y_r$ since $\{T\in C_r\}\subseteq\{T\ge (r-1)f+1\}$, so indeed $x_r=0$ if $y_r=0$. Then,
\[
\sum_{r\ge1} y_r r_r=\sum_{r\ge1} x_r=1.
\]
Restricting the average to indices with $y_r>0$ gives us,
\[
\sum_{r:\,y_r>0} \underbrace{\frac{y_r}{\sum_{j:\,y_j>0}y_j}}_{\text{weights}}\ r_r\ =\
\frac{1}{\sum_{j:\,y_j>0}y_j}\ \le\ \frac{1}{\sum_{j\ge1}y_j}\ \le\ \frac{f}{\gamma}.
\]
Therefore $\min_{r:\,y_r>0} r_r\le f/\gamma$. Pick $r^\star$ attaining this minimum; then $y_{r^\star}>0$ and $x_{r^\star}/y_{r^\star}\le f/\gamma$.
\end{proof}

\begin{proof}[Proof of Lemma~\ref{lem:par-asymp}]
First, write,
\[
x_\gamma:=\frac{(1-\varepsilon)\log\gamma}{I_\delta}, \qquad f_\gamma=x_\gamma-\theta_\gamma\quad \text{with } \theta_\gamma\in[0,1).
\]
\smallskip
Now, since $\log\gamma\to\infty$ as $\gamma\to\infty$, we have that,
\[
\frac{f_\gamma}{\log\gamma}
=\frac{x_\gamma-\theta_\gamma}{\log\gamma}
=\frac{1-\varepsilon}{I_\delta}-\frac{\theta_\gamma}{\log\gamma}
\;\xrightarrow[\gamma\to\infty]{}\; \frac{1-\varepsilon}{I_\delta},
\]
which follows from the fact that $0\le\theta_\gamma<1$ implies $\theta_\gamma/\log\gamma\to0$. Applying this above result,
\[
\frac{c_\gamma}{f_\gamma}
=\frac{b\log\gamma}{f_\gamma}
=\frac{b}{\;f_\gamma/\log\gamma\;}
\;\xrightarrow[\gamma\to\infty]{}\;
\frac{b}{(1-\varepsilon)/I_\delta}
=\frac{b}{1-\varepsilon}\,I_\delta.
\]
We assumed that $b>\frac{(1-\varepsilon)\mu_\delta}{I_\delta}$, so multiplying both sides by $\frac{I_\delta}{1-\varepsilon}$ gives us,
\[
\frac{b}{1-\varepsilon}\,I_\delta \;>\; \mu_\delta.
\]
\smallskip
Let $L:=\frac{b}{1-\varepsilon}\,I_\delta$; then $L>\mu_\delta$. Set $\eta:=(L-\mu_\delta)/2>0$. As we just showed, $\frac{c_\gamma}{f_\gamma}\to L$ clearly, so there exists $\gamma_0$ such that $\big|\frac{c_\gamma}{f_\gamma}-L\big|<\eta$ for all $\gamma\ge\gamma_0$. So it follows that,
\[
\frac{c_\gamma}{f_\gamma}\ \ge\ L-\eta \;=\; \frac{L+\mu_\delta}{2} \;=\; \mu_\delta+\eta
\qquad(\gamma\ge\gamma_0).
\]
It remains to show now that the term $e^{c_\gamma}f_\gamma/\gamma$ tends to $0$. Since $c_\gamma=b\log\gamma$, $e^{c_\gamma}=\gamma^b$, hence,
\[
\frac{e^{c_\gamma}f_\gamma}{\gamma}
=f_\gamma\,\gamma^{\,b-1}.
\]
Using the fact that $f_\gamma\le x_\gamma=\frac{1-\varepsilon}{I_\delta}\log\gamma$ for all $\gamma$,
\[
0\ \le\ \frac{e^{c_\gamma}f_\gamma}{\gamma}
\ \le\ \frac{1-\varepsilon}{I_\delta}\,(\log\gamma)\,\gamma^{\,b-1}
=\frac{1-\varepsilon}{I_\delta}\,\frac{\log\gamma}{\gamma^{\,1-b}}.
\]

\noindent From here, clearly we see that $\frac{\log\gamma}{\gamma^{\,1-b}}\to0$ and hence we have shown that $e^{c_\gamma}f_\gamma/\gamma\to0$ so we are done.
\end{proof}

\begin{proof}[Proof of Corollary~\ref{cor:minimax-cadd-lb}]
Take an arbitrary $Q\in\mathcal Q$ with $0<I(Q;\mathcal P)<\infty$. By Theorem~\ref{thm:cadd}, we have that,
\[
\liminf_{\gamma\to\infty}\ \frac{I(Q;\mathcal P)\,\mathcal C_Q(T_\gamma)}{\log\gamma}\ \ge\ 1.
\]
\noindent We know that for each $\gamma$, the supremum over $Q\in\mathcal Q$ dominates this value at fixed $Q$. Hence,
\[
\sup_{Q\in\mathcal Q}\frac{I(Q;\mathcal P)\,\mathcal C_Q(T_\gamma)}{\log\gamma} \ \ge\ \frac{I(Q;\mathcal P)\,\mathcal C_Q(T_\gamma)}{\log\gamma}.
\]

\noindent All we need to do now is take $\liminf_{\gamma\to\infty}$ and we get exactly our claim, hence we are done.
\end{proof}

\section{Ommitted Proofs for Section~\ref{sec:bounded-mean}}\label{sec:omittedproofs-two}
\begin{definition}\label{def:TV}
For two probability measures $P,Q$ on $(\mathcal X,\mathcal A)$, define the total variation distance by $\|P-Q\|_{\mathrm{TV}}:=\sup_{A\in\mathcal A}|P(A)-Q(A)|$.
\end{definition}
\begin{proof}[Proof of Lemma~\ref{lem:bm-pos-fin}]
\underline{We will first show that the KL is finite.} We know that since $q:=\EE_Q[X_1]>m$ by definition, we can choose $\varepsilon\in(0,1)$ such that $(1-\varepsilon)q\le m$. Let's define the mixture $P_\varepsilon := (1-\varepsilon)Q + \varepsilon\delta_0$. Of course then we know based on how we chose $\varepsilon$ that $\EE_{P_\varepsilon}[X_1]=(1-\varepsilon)\EE_Q[X_1]\le m$. Thus by definition, the mixture $P_\varepsilon\in\mathcal P_m$. In addition, for any measurable $A$ we have that $P_\varepsilon(A)=(1-\varepsilon)Q(A)+\varepsilon\delta_0(A) \ge (1-\varepsilon)Q(A)$. As a consequence, $Q(A)\le \frac{1}{1-\varepsilon}P_\varepsilon(A)$ for all $A$. So therefore by the Radon-Nikodym theorem, this domination inequality is in fact equivalent to $Q\ll P_\varepsilon$ and thus $\frac{dQ}{dP_\varepsilon}\le \frac{1}{1-\varepsilon}$ $Q$-almost-surely. Therefore, we get that,
\[
\KL(Q\|P_\varepsilon) =\EE_Q\left[\log\Big(\frac{dQ}{dP_\varepsilon}\Big)\right] \le \log\Big(\frac{1}{1-\varepsilon}\Big) <\infty.
\]
\noindent All we need to do now is take the infimum over all $P\in\mathcal P_m$ and we get that $\mathrm{KL}_{\inf}(Q;m)\le \KL(Q\|P_\varepsilon)<\infty$. \underline{We will now show strict positivity.} Let's prove this by contradiction. Meaning, assume to the contrary that $\mathrm{KL}_{\inf}(Q;m)=0$. This means that there exists a sequence $(P_n)_{n\ge 1}\subset\mathcal P_m$ such that $\KL(Q\|P_n)\to 0$. We can now use Pinsker's inequality to get that $\|Q-P_n\|_{\mathrm{TV}}\le \sqrt{\frac{\KL(Q\|P_n)}{2}}\longrightarrow 0$. Now, since $X_1\in[0,1]$ it's easy to see that $|\EE_Q[X_1]-\EE_{P_n}[X_1]|\le \|Q-P_n\|_{\mathrm{TV}}\longrightarrow0$. As such, $\EE_{P_n}[X_1]\to \EE_Q[X_1]=q>m$, which contradicts the fact that $\EE_{P_n}[X_1]\le m$ for all $n$. As we know that $P_n\in\mathcal P_m$. Therefore it must be the case that our initial supposition was false! Hence, indeed $\mathrm{KL}_{\inf}(Q;m)>0$.

\underline{We now prove the uniform gap statement.} Take a $Q\in\mathcal Q_{m+\Delta}$ and $P\in\mathcal P_m$. Then by definition we'd have that $\EE_Q[X_1]\ge m+\Delta$ and $\EE_P[X_1]\le m$. As such, we get that $\EE_Q[X_1]-\EE_P[X_1]\ge \Delta$. It is easy to see that $\EE_Q[X_1]-\EE_P[X_1]\le |\EE_Q[X_1]-\EE_P[X_1]|\le \|Q-P\|_{\mathrm{TV}}$. Hence $\|Q-P\|_{\mathrm{TV}}\ge \Delta$. Let's now apply Pinsker's inquality. If we do this, we get that $\Delta \le \|Q-P\|_{\mathrm{TV}}\le \sqrt{\frac{\KL(Q\|P)}{2}}$. As such, we get that $\KL(Q\|P)\ge 2\Delta^2$. Let's now take the infimum over $P\in\mathcal P_m$ to give us that $\mathrm{KL}_{\inf}(Q;m)\ge 2\Delta^2$. Then, taking the infimum over all $Q\in\mathcal Q_{m+\Delta}$ gives us exactly our lemma's claim for that part and hence we are done. 

\underline{It remains to now prove that $\mathrm{KL}_{\inf}(Q;m)= \sup_{\lambda\in[0,1]}\EE_Q[\log L_\lambda(X_1)]$.} Note that really this is well known in the literature, extensively showed by \citep{HondaTakemura2010}. However, we present a self contained proof here as we mentioned. Let us show the another fact first. Take some $\lambda\in(0,1)$ and $P\in\mathcal P_m$ with $Q\ll P$. Define the likelihood ratio shorthand as $r:=\frac{dQ}{dP}$. Since $X_1\in[0,1]$ and $\EE_P[X_1]\le m$, as we argued above, we necessarily have that $\EE_P[L_\lambda(X_1)]=1+\lambda\Big(\frac{\EE_P[X_1]}{m}-1\Big)\le 1$, obviously. Now we of course know that,
\begin{align*}
\KL(Q\|P)-\EE_Q[\log L_\lambda(X_1)] &=\EE_Q\!\left[\log\!\Big(\frac{r(X_1)}{L_\lambda(X_1)}\Big)\right] =\EE_P\!\left[r(X_1)\log\!\Big(\frac{r(X_1)}{L_\lambda(X_1)}\Big)\right].
\end{align*}

\noindent From here, we can easily use the log-sum inequality, or the fact that nonegativity of KL divergence holds to give us that,
\[
\EE_P\!\left[r\log\!\Big(\frac{r}{L_\lambda}\Big)\right] \ge \Big(\EE_P[r]\Big)\log\!\Big(\frac{\EE_P[r]}{\EE_P[L_\lambda]}\Big) =1\cdot\log\!\Big(\frac{1}{\EE_P[L_\lambda]}\Big) \ge 0,
\]
because we also know that $\EE_P[r]=\int \frac{dQ}{dP}dP=\int dQ=1$, and $\EE_P[L_\lambda]\le 1$. Hence, from all this we can easily conclude that \underline{$\KL(Q\|P)\ge \EE_Q[\log L_\lambda(X_1)]$.} Now let us return back to the task of showing $\mathrm{KL}_{\inf}(Q;m) = \sup_{\lambda\in[0,1]}\EE_Q[\log L_\lambda(X_1)]$. Our proof will proceed through splitting the task into showing two inequalities. First, we will prove the $\ge$ direction, and we will do so by first proving it excluding the boundaries and then handling the boundaries after. First, take some $\lambda\in(0,1)$. As we just proved, $\KL(Q\|P)\ge \EE_Q[\log L_\lambda(X_1)]$. Then, taking the infimum over $P\in\mathcal P_m$ gives us that $\mathrm{KL}_{\inf}(Q;m)\ge \EE_Q[\log L_\lambda(X_1)]$. Then taking the supremum over all $\lambda\in(0,1)$ gives us that $\mathrm{KL}_{\inf}(Q;m)\ge \sup_{\lambda\in(0,1)}\EE_Q[\log L_\lambda(X_1)]$. Let's now look at the boundaries. Since $L_0\equiv1$, we get that $\EE_Q[\log L_0(X_1)]=0,$ and since $0$ is a finite value it can indeed be captured by the supremum over the closed interval $(0,1)$. For the other case, at $\lambda=1$, $\log L_1(X_1)=\log(X_1/m)$ may indeed be $-\infty$ for two reasons: firstly, because $Q{\{0\}}>0$ (ie, $Q$ can take value $0$ with nonzero probability), or because $\EE_Q[\log X_1]=-\infty$, but in either case it cannot increase the supremum of course. Contrarily, if $\EE_Q[\log L_1(X_1)]>-\infty$, then by the dominated convergence theorem along $\lambda_n\uparrow 1$, we get that $\EE_Q[\log L_{\lambda_n}(X_1)]\to \EE_Q[\log L_1(X_1)]$, so that endpoint value is again captured by the supremum over $(0,1)$. Hence it must be the case that $\mathrm{KL}_{\inf}(Q;m)\ge \sup_{\lambda\in[0,1]}\EE_Q[\log L_\lambda(X_1)]$.

The rest of this proof will be dedicated to \underline{proving the $\le$ direction} so that we can achieve the desired equality. We will prove it through convex duality. Meaning, we will write the minimization defining $\mathrm{KL}_{\inf}(Q;m)$ as a convex program and then compute its Lagrange dual. To start, let us reduce the primal problem into densities that are with respect to $Q$. Consider some $P\in\mathcal P_m$ such that $\KL(Q\|P)<\infty$, so then necessarily we know that $Q\ll P$. Now, let $P^{\mathrm{ac}}$ be the absolutely continuous part of $P$ with respect to $Q$ in the Lebesgue decomposition $P=P^{\mathrm{ac}}+P^{\mathrm{s}}$, where here $P^{\mathrm{s}}\perp Q$ (ie, singular to $Q$). Further, let us define the shorthand $p:=\frac{dP^{\mathrm{ac}}}{dQ}$. Because of the fact that $P^{\mathrm{s}}\perp Q$, by definition of singularity, there must exist a measurable set $N$ where $Q(N)=0$ and $P^{\mathrm{s}}(N^c)=0$. Now, if $P^{\mathrm{ac}}(A)=0$, then $P(A\cap N^c)=P^{\mathrm{ac}}(A\cap N^c)+P^{\mathrm{s}}(A\cap N^c)=0$. So, we get that $Q(A\cap N^c)=0$ by the fact that $Q\ll P$. Now, notice that we can write $A=(A\cap N^c)\cup(A\cap N)$. We just showed that $Q(A\cap N^c)=0$. And, the set $A\cap N$ is a subset of $N$ of course, so $Q(A\cap N)$ must be $0$ since $Q(N)=0$. Since $Q(A)$ is nothing but a measure of the union of these two null-sets, it follows that $Q(A)$ is zero also. Therefore, $Q\ll P^{\mathrm{ac}}$ as we have proven that any set $A$ such that $P^{\mathrm{(ac)}}(A)=0$ implies that $Q(A)=0$.  It thus follows that $p>0$ $Q$-almost-surely. In addition on the $Q$-support, we have that $\frac{dQ}{dP}=\frac{dQ}{dP^{\mathrm{ac}}}=1/p$ as the singular part drops out. In other words, these derivatives are equal wherever $Q$ has mass. Therefore, we get that,
\[
\KL(Q\|P)=\EE_Q\!\left[\log\!\Big(\frac{dQ}{dP}(X_1)\Big)\right] =\EE_Q[-\log p(X_1)].
\]
\noindent Let us denote $c:=\EE_Q[p(X_1)]=P^{\mathrm{ac}}([0,1])\le 1$. By the mean-constraint on the nulls we know that,
\[
\EE_Q[X_1p(X_1)] =\int x\,P^{\mathrm{ac}}(dx) \le \int x\,P(dx)=\EE_P[X_1]\le m,
\]
since $P^{\mathrm{s}}$ contributes a nonnegative amount to $\EE_P[X_1]$ because recall that $X_1\in[0,1]$. What this means is that every feasible null $P$ induces a function $p$ satisfying four properties. First, that $p>0$ $Q$-almost-surely; second, that $\EE_Q[p(X_1)]\le 1$; third, that $\EE_Q[X_1p(X_1)]\le m$; and finally, that $\KL(Q\|P)=\EE_Q[-\log p(X_1)]$. 

\medskip
\noindent Conversely, consider some fixed measurable $p>0$, $Q$-almost-surely with $\EE_Q[p(X_1)]\le 1$ and $\EE_Q[X_1p(X_1)]\le m$. Now just so we can avoid any nuance when the mean constraint is indeed tight, let us take some $\delta\in(0,1)$ and define $p_\delta:=(1-\delta)p$. Then we know that $\EE_Q[p_\delta(X_1)]<1$ and $\EE_Q[X_1p_\delta(X_1)]<(1-\delta)m<m$. Let us now introduce ``slack'' variables for any remaining mass and mean difference needed to reach the full constraints: that is, let the mass slack be $s_\delta:=1-\EE_Q[p_\delta(X_1)]>0$, and mean slack be $\mathrm{slack}_\delta:=m-\EE_Q[X_1p_\delta(X_1)]>0$. We know that $Q$ has countable many atoms, by definition of $\sigma$-finite measure (which any probability measure must be). Because of this fact that there exist uncountably many atomless points, we know that for every $\varepsilon>0$ there must exist some $x\in(0,\varepsilon)$ with $Q(\{x\})=0$. To this end, let us choose some singular point $x_\delta>0$ such that $Q(\{x_\delta\})=0$ and $x_\delta \le \frac{\mathrm{slack}_\delta}{s_\delta}$. Now define $P_\delta(dx):=p_\delta(x)Q(dx)+s_\delta\delta_{x_\delta}(dx)$. Clearly, $P_\delta$ is indeed a probability measure. In addition, we have that,
\[
\EE_{P_\delta}[X_1] =\EE_Q[X_1p_\delta(X_1)] + s_\delta x_\delta \le \EE_Q[X_1p_\delta(X_1)] + \mathrm{slack}_\delta = m,
\]
so of course $P_\delta\in\mathcal P_m$. Finally, since $\delta_{x_\delta}$ is mass supported on a $Q$-null set, it's not going to affect $\frac{dQ}{dP_\delta}$ on the $Q$-support. Therefore, we have that,
\[
\KL(Q\|P_\delta)=\EE_Q[-\log p_\delta(X_1)] =\EE_Q[-\log p(X_1)]-\log(1-\delta).
\]
\noindent Now, if we just let $\delta\downarrow 0$ we can see that optimizing over all $P\in\mathcal P_m$ is equivalent to optimizing over such $p$. Thus, beautifully, we can rewrite our primal problem as,
\begin{equation}\label{eq:bm-primal}
\mathrm{KL}_{\inf}(Q;m) =\inf\left\{ \EE_Q[-\log p(X_1)]: \begin{array}{l} p>0\ \ Q\text{-a.s.},\\ \EE_Q[p(X_1)]\le 1,\\ \EE_Q[X_1p(X_1)]\le m \end{array} \right\}.   
\end{equation}

\noindent Clearly, this is a convex optimization problem in the variable $p$, since $p\mapsto -\log p$ is convex and the constraints are indeed linear. Furthermore, we know that Slater's condition of strict feasibility holds. Why? To see this, we can easily proceed as follows. Since $q=\EE_Q[X_1]>m$, choose any constant $c\in(0, m/q)$. Clearly, $p\equiv c$ will satisfy $p>0$, $\EE_Q[p]=c<1$, and $\EE_Q[X_1p]=cq<m$. Therefore, \eqref{eq:bm-primal} indeed is strictly feasible. Therefore, there is no duality gap and the dual optimum solution equals the primal optimal solution. As such, we are now ready to compute the Lagrange dual to solve this convex optimization problem. To begin, let us introduce multipliers $\alpha\ge 0$ and $\beta\ge 0$ for the constraints $\EE_Q[p]\le 1$ and $\EE_Q[X_1p]\le m$ respectively. Then, the Lagrangian becomes,
\[
\mathcal L(p;\alpha,\beta) :=\EE_Q\!\left[-\log p(X_1)+(\alpha+\beta X_1)p(X_1)\right]-\alpha-\beta m.
\]
\noindent Now, for some particular $(\alpha, \beta)$ pair, we minimize the Lagrangian $\mathcal L$ over $p>0$ pointwise: to begin, take some $x\in[0,1]$ and set $c(x):=\alpha+\beta x$. Consider $\varphi_x(p):= -\log p + c(x)p$ (where $p>0$ as above). If in fact $c(x)\le 0$, then necessarily that would mean that $\inf_{p>0}\varphi_x(p)=-\infty$ (send $p\to\infty$). So, we will restrict to parameters such that $c(X_1)>0$ $Q$-almost-surely; and in fact, in our setting it suffices to assume that $\alpha>0$, or more generally that $\alpha+\beta x>0$ $Q$-almost-surely. To this end, when $c(x)>0$, the derivative $\varphi_x'(p)=-1/p + c(x)$ will vanish at $p^\star(x)=1/c(x)$. So, the minimum value is simply $\varphi_x(p^\star(x))=\log c(x)+1$. Thus the dual function becomes,
\[
\inf_{p>0}\mathcal L(p;\alpha,\beta)=1+\EE_Q[\log(\alpha+\beta X_1)]-\alpha-\beta m,
\]
and the corresponding dual problem is simply,
\begin{equation}\label{eq:bm-dual-raw}
\sup_{\alpha\ge0,\ \beta\ge0:\, \alpha+\beta X_1>0\ Q\text{-a.s.}} \Big\{1+\EE_Q[\log(\alpha+\beta X_1)]-\alpha-\beta m\Big\}.  
\end{equation}

\noindent And, as we already explained, there is no duality gap, which therefore entails that the value of \eqref{eq:bm-dual-raw} equals that of the $\mathrm{KL}_{\inf}(Q;m)$. We will now conclude our proof by reducing \eqref{eq:bm-dual-raw} to the one-parameter family $L_\lambda$. Meaning, take any particular $(\alpha,\beta)$ with $\alpha>0$ and $\beta\ge 0$ and consider the scaling by $t>0$, $(\alpha,\beta)\mapsto (t\alpha,t\beta)$. So, the dual objective becomes,
\begin{align*}
1+\EE_Q[\log(t\alpha+t\beta X_1)] - t\alpha -t\beta m &=1+\log t + \EE_Q[\log(\alpha+\beta X_1)]-t(\alpha+\beta m).    
\end{align*}

\noindent We know that when we consider a particular $(\alpha,\beta)$ pair, the right hand side is nothing but a concave function of $t>0$, and is therefore maximized when $\frac{d}{dt}(\log t-t(\alpha+\beta m))=\frac{1}{t}-(\alpha+\beta m)=0$, or more simply when $t^\star = \frac{1}{\alpha+\beta m}$. Then, substituting $t^\star$ back in will give us the maximized value $\EE_Q[\log(\alpha + \beta X_1)]-\log(\alpha+\beta m)$. So, without loss of optimality we can restrict our attention to normalized parameters that satisfy $\alpha+\beta m=1$. Under such a normalization, the dual objective will reduce to $\EE_Q[\log(\alpha+\beta X_1)]$. Now, let us parameterize $\alpha+\beta m=1$ with $\lambda:=\beta m\in[0,1]$, so $\alpha=1-\lambda$, and $\beta = \lambda/m$. We then get that $\alpha+\beta X_1 = (1-\lambda)+\frac{\lambda}{m}X_1=L_\lambda(X_1)$. And, the dual value becomes $\sup_{\lambda\in[0,1]}\EE_Q[\log L_\lambda(X_1)]$. Of course since the dual optimum equals that of the primal, we obtain that,
\[
\mathrm{KL}_{\inf}(Q;m) \le \sup_{\lambda\in[0,1]}\EE_Q[\log L_\lambda(X_1)],
\]
and if we combine this with the first part of our proof, we indeed get equality and hence this concludes the proof. As we have shown all parts of the lemma, we are now done. 
\end{proof}

\begin{proof}[Proof of Lemma~\ref{lem:bm-hit}]
The first thing we need to show is that $\tau_u<\infty$ almost surely (ie, non-defective). Obviously we know that $\EE[|Y_1|]<\infty$ by boundedness. In addition, we assumed that $\EE[Y_1]=\mu>0$. By the strong law of large numbers, we know that $S_n/n\to\mu>0$ almost surely. Thus, $S_n\to+\infty$ almost surely. Therefore because $S_n$ is guaranteed to at some point always exceed the threshold $u$, we know that $\tau_u<\infty$ almost surely for every fixed $u>0$. 

We now need to establish that $\EE[\tau_u]<\infty$. Take some $u>0$ and define the bounded stopping rules $\tau_u^{(N)}:=\tau_u\wedge N$. We know that the process $M_n:=S_n -\mu n$ is a martingale with respect to the filtration $\mathcal F_n:=\sigma(Y_1,Y_2,\dots,Y_n)$. By optional stopping at the bounded time $\tau_u^{(N)}$, we have that $\EE[S_{\tau_u^{(N)}}-\mu \tau_u^{(N)}]= \EE[M_{\tau_u^{(N)}}] = \EE[M_0]=0$. As such, $\EE[S_{\tau_u^{(N)}}]=\mu\EE[\tau_u^{(N)}]$. Let's now bound $S_{\tau_u^{(N)}}$ from above. On the event where $\{\tau_u\le N\}$, we know that $S_{\tau_u - 1}<u$ and $Y_{\tau_u}\le b$, so $S_{\tau_u}=S_{\tau_u-1}+Y_{\tau_u}<u+b$. Now on the event where $\{\tau_u>N\}$, by definition of $\tau_u$ we have that $S_N<u$. Hence, $S_{\tau_u^{(N)}}=S_N<u<u+b$. In any case, we can be certain thus that with probability one $S_{\tau_u^{(N)}}<u+b$. Quickly taking expectations gives us that $\EE[S_{\tau_u^{(N)}}]\le u+b$ which means that for all $N$,
\[
\EE[\tau_u^{(N)}] = \frac{\EE[S_{\tau_u^{(N)}}]}{\mu} \le \frac{u+b}{\mu}.
\]
\noindent Now, if we let $N\to\infty$ and apply the monotone convergence theorem on $\tau_u^{(N)}\uparrow \tau_u$ we get that,
\[
\EE[\tau_u] = \lim_{N\to\infty}\EE[\tau_u^{(N)}]\le \frac{u+b}{\mu} < \infty.
\]
\noindent Thus, we have proven that the expectation is indeed finite: $\EE[\tau_u]<\infty$. We will now extend the result we have shown here from bounded stopping times to the actual unbounded stopping time $\tau_u$. In addition, we know that the martingale differences of $M_n=S_n-\mu n$ are bounded. That is, $|M_n-M_{n-1}|=|Y_n-\mu|\le \max\{|a-\mu|,|b-\mu|\}:=C<\infty$. So it follows that for each $N$, $|M_{\tau_u^{(N)}}|\le C\tau_u^{(N)}$ almost surely. Hence for each $N$ we know that $\sup_N \EE[|M_{\tau_u^{(N)}}|]\le C\EE[\tau_u]<\infty$. By the dominated convergence theorem, we know that $M_{\tau_u^{(N)}}\to M_{\tau_u}$ in $L^1$. Thus we get that $\EE[M_{\tau_u}]=\lim_{N\to\infty}\EE[M_{\tau_u^{(N)}}]=0$. Equivalently,
\begin{equation}\label{eq:wald-hit}
\EE[S_{\tau_u}] = \mu \EE[\tau_u].  
\end{equation}

\noindent Lastly, we know that the overshoot is also bounded. By definition $S_{\tau_u}\ge u$ and also $S_{\tau_u}<u+b$ (immediate by using our same above argument). Hence it follows that $u\le \EE[S_{\tau_u}]\le u+b$. Using \eqref{eq:wald-hit} gives us therefore that,
\[
\frac{u}{\mu} \le \EE[\tau_u]\le \frac{u+b}{\mu}.
\]
\noindent Now all we need to do is divide by $u$ and let $u\to\infty$ and we get by the squeeze theorem that $\EE[\tau_u]/u\to 1/\mu$, and so we are done.
\end{proof}

\begin{proof}[Proof of Theorem~\ref{thm:bm-opt}: Statements $1$ and $3$.]
We start with the first statement. Consider some particular but arbitrarily chosen $P\in\mathcal P_m$. First, we will show that for each $\lambda\in\Lambda$, the process $R^{(\lambda)}_n - n$ is a supermartingale under $P^\infty$. Indeed we know that by \eqref{eq:bm-SR-rec} and the tower property of expectation we get that,
\begin{align*}
\EE_{P^\infty}\!\left[R^{(\lambda)}_n\mid \mathcal F_{n-1}\right] &=\big(1+R^{(\lambda)}_{n-1}\big)\,\EE_{P^\infty}\!\left[L_\lambda(X_n)\mid \mathcal F_{n-1}\right].
\end{align*}

\noindent We know that under $P^\infty$, $X_n$ is independent of the sigma algebra $\mathcal F_{n-1}$, and that $\EE_P[X_1]\le m$, hence we get that,
\[
\EE_{P^\infty}[L_\lambda(X_n)\mid \mathcal F_{n-1}] =\EE_P[L_\lambda(X_1)] =1+\lambda\Big(\frac{\EE_P[X_1]}{m}-1\Big) \le 1.
\]
\noindent Hence it follows that $\EE_{P^\infty}[R^{(\lambda)}_n \mid \mathcal F_{n-1}] \le 1+R^{(\lambda)}_{n-1}$. Equivalently, $\EE_{P^\infty}[R^{\lambda}_n-n\mid \mathcal F_{n-1}] \le R^{(\lambda)}_{n-1}-(n-1)$. So, we know that $R^{(\lambda)}_n-n$ is indeed a supermartingale. Taking the convex combination gives us that $M_n-n=\sum_{\lambda\in\Lambda} w_\lambda R^{(\lambda)}_n-n$, which is also a supermartingale under $P^\infty$. Now let $T:=T^{\mathrm{BM}}_\gamma=\inf\{n\ge1:M_n\ge\gamma\}$. A defective stopping rule (namely, $T=\infty$) would imply that $\EE_{P^\infty}[T]\ge\gamma$ is immediate. So, wlog we can assume that $\EE_{P^\infty}[T]<\infty$. In particular, let us assume that $T<\infty$ almost surely. If for each $N\in\NN$, we apply optional stopping to the bounded stopping rule $T\wedge N$, we can see that $\EE_{P^\infty}[M_{T\wedge N}-(T\wedge N)]\le M_0 -0=0$. So, $\EE_{P^\infty}[M_{T\wedge N}]\le \EE_{P^\infty}[T\wedge N]$. Since $T<\infty$ almost surely, we have that $T\wedge N\uparrow T$ almost surely as $N\to\infty$. Furthermore, we know that $(M_n)_{n\ge 0}$ is adapted, and so this implies that $M_{T\wedge N}\to M_T$ almost surely as $N\to\infty$. We know also that $M_{T\wedge N}\ge 0$. Hence we can apply Fatou's lemma to get that,
\[
\EE_{P^\infty}[M_T] = \EE_{P^\infty}\!\left[\lim_{N\to\infty}M_{T\wedge N}\right] \le \liminf_{N\to\infty}\EE_{P^\infty}[M_{T\wedge N}] \le \liminf_{N\to\infty}\EE_{P^\infty}[T\wedge N] = \EE_{P^\infty}[T].
\]

\noindent Finally, by definition of $T$, we know that on $\{T<\infty\}$, we have that $M_T\ge\gamma$. Hence it follows that $\EE_{P^\infty}[T]\ge \EE_{P^\infty}[M_T]\ge \gamma$. Note that $P\in\mathcal P_m$ was arbitrary. So indeed we can conclude that \underline{$\inf_{P\in\mathcal P_m}\EE_{P^\infty}[T^{\mathrm{BM}}_\gamma]\ge\gamma$}, which completes our the proof of the first statement. 

\underline{It remains to now prove the last statement of the theorem.} Let's do it. To begin, take a $\Delta\in(0,1-m)$ and consider $Q\in\mathcal Q_{m+\Delta}$. Just to recap, we take $I_Q=\mathrm{KL}_{\inf}(Q;m)\in(0,\infty)$. We know that necessarily by Lemma~\ref{lem:bm-pos-fin} we have the uniform lower bound that $I(Q)\ge \underline I_\Delta := 2\Delta^2$ for all $Q\in\mathcal Q_{m+\Delta}$. Now, for $\lambda\in[0,1]$, let us define the function $g_Q(\lambda):=\EE_Q[\log L_\lambda(X_1)]$. Of course we know that by Lemma~\ref{lem:bm-pos-fin}, $I_Q=\sup_{\lambda\in[0,1]}g_Q(\lambda)$ and that $g_Q(0)=0$. Now take an arbitrary $\varepsilon\in(0,1/2)$. We're first going to show that there indeed exists a finite subset of betting fractions $\Lambda_{\Delta,\varepsilon}\subset\Lambda$ and a constant $w_{\Delta,\varepsilon}>0$ such that for every $Q\in\mathcal Q_{m+\Delta}$, there exists some $\lambda(Q)\in\Lambda_{\Delta,\varepsilon}$ where,
\begin{equation}\label{eq:unif-goodlambda}
g_Q(\lambda(Q))\ge (1-\varepsilon)I_Q \quad\text{and}\quad w_{\lambda(Q)}\ge w_{\Delta,\varepsilon}.  
\end{equation}
Let us now build this $\Lambda_{\Delta,\varepsilon}$. First define $\varepsilon':=\varepsilon/2\in(0,1/4)$. We know that $\log(1+u)\le u$ for $u>-1$. Hence, for any $\lambda\in[0,1]$ and $x\in[0,1]$, we have that $\log L_\lambda(x)=\log\Big(1+\lambda\left(\frac{x}{m}-1\right)\Big) \le \lambda\Big(\frac{x}{m}-1\Big) \le \lambda\Big(\frac{1}{m}-1\Big)$. Therefore, we get that for all $\lambda\in[0,1]$,
\begin{equation}\label{eq:g-upperlinear}
g_Q(\lambda)\le \lambda\Big(\frac{1}{m}-1\Big).
\end{equation}
Now take a particular $Q\in\mathcal Q_{m+\Delta}$. Since $I_Q=\sup_{\lambda\in[0,1]}g_Q(\lambda)$,  we can choose measurably a point $\lambda_Q\in(0,1)$ such that,
\begin{equation}\label{eq:approx-max}
g_Q(\lambda_Q)\ge (1-\varepsilon')I_Q.  
\end{equation}
If we combine \eqref{eq:approx-max} with \eqref{eq:g-upperlinear} we get the lower bound that $(1-\varepsilon')I(Q)\le g_Q(\lambda_Q)\le \lambda_Q\Big(\frac{1}{m}-1\Big)$. Hence, it follows that,
\[
\lambda_Q \ge (1-\varepsilon')\frac{I_Q}{\frac{1}{m}-1} \ge (1-\varepsilon')\frac{\underline I_\Delta}{\frac{1}{m}-1} =: \underline\lambda_{\Delta,\varepsilon} >0.
\]
Now let's set the mesh size as $\delta:=\frac{\varepsilon'\underline\lambda_{\Delta,\varepsilon}}{4}$. Now let $J:=\lfloor 1/\delta\rfloor$ so that $J\delta\le 1<(J+1)\delta$. For each of the $j\in\{1,\dots, J\}$ since $\Lambda$ is dense in $(0,1)$ we can easily choose a $\lambda_j\in \Lambda\cap\big(j\delta-\delta/2,\ j\delta\big]$. Furthermore, let us also take and pick one more point near $1$, $\lambda_{J+1}\in \Lambda\cap(1-\delta/2,\ 1)$. Having done this, let's define the finite set,
\[
\Lambda_{\Delta,\varepsilon}:=\{\lambda_1,\dots,\lambda_{J+1}\}\subset\Lambda, \qquad w_{\Delta,\varepsilon}:=\min_{\lambda\in\Lambda_{\Delta,\varepsilon}} w_\lambda >0.
\]
At this point having constructed the set, let us now take a particular but arbitrarily chosen $Q\in\mathcal Q_{m+\Delta}$ and consider a $\lambda_Q$ from \eqref{eq:approx-max}. Let $j^\star:=\lfloor \lambda_Q/\delta\rfloor$, so it follows that $j^\star\ge 1$ because we had that $\lambda_Q\ge \underline \lambda_{\Delta,\varepsilon}>\delta$. In addition, $j^\star\le J$ necessarily also holds because $\lambda_Q<1$ and $J=\lfloor 1/\delta \rfloor$. Further, by construction we have that $\lambda_{j^\star}\le j^\star\delta\le \lambda_Q$ and $\lambda_{j^\star}\ge j^\star\delta-\frac{\delta}{2}\ge \lambda_Q-\delta-\frac{\delta}{2}=\lambda_Q-\frac{3\delta}{2}$. Now, we know that $\delta=\varepsilon'\underline\lambda_{\Delta,\varepsilon}/4 \le \varepsilon'\lambda_Q/4$. As a consequence of this, we have that $\frac{3\delta}{2}\le \varepsilon'\lambda_Q$. Hence, we get that $\frac{\lambda_{j^\star}}{\lambda_Q}\ge 1-\varepsilon'$. By concavity of $g_Q$ alongside the fact that $g_Q(0)=0$ we get that for any $t\in[0,1]$ we have that $g_Q(t\lambda_Q)\ge t g_Q(\lambda_Q)$. Considering $t=\lambda_{j^\star}/\lambda_Q\in[1-\varepsilon',1]$ gives us,
\[
g_Q(\lambda_{j^\star})\ge \frac{\lambda_{j^\star}}{\lambda_Q}g_Q(\lambda_Q)\ge (1-\varepsilon')(1-\varepsilon')I_Q \ge (1-\varepsilon)I_Q,
\]
where the last inequality holds since $(1-\varepsilon')^2 \ge 1-\varepsilon$ whenever $\varepsilon'=\varepsilon/2$. Also of course we know that $w_{\lambda_{j^\star}}\ge w_{\Delta,\varepsilon}$ by construction. Therefore, \eqref{eq:unif-goodlambda} indeed holds with $\lambda(Q):=\lambda_{j^\star}\in\Lambda_{\Delta,\varepsilon}$. Let us now use this good component to bound the CADD uniformly as per the theorem statement. As we did before, consider a particular $Q\in\mathcal Q_{m+\Delta}$ and let $\lambda(Q)$ be defined as above. We know that $M_n\ge w_{\lambda(Q)}R_n^{\lambda(Q)}\ge w_{\Delta,\varepsilon}R_n^{(\lambda(Q))}$. Hence we have the pathwise domination,
\[
T^{\mathrm{BM}}_\gamma \le \inf\Big\{n\ge 1: R_n^{(\lambda(Q))}\ge \gamma/w_{\Delta,\varepsilon}\Big\} =: T^{(\lambda(Q))}_{\gamma/w_{\Delta,\varepsilon}}.
\]
We can now repeat our same post-change product argument that we used in our above pointwise argument. Doing this tells us that for every change time $k\ge 1$, $\mathcal C_Q(T^{\mathrm{BM}}_\gamma)\le \EE_Q\left[\tau_{u_\gamma}\right]$ and $u_\gamma:=\log(\gamma/w_{\Delta,\varepsilon})$. Here, $\tau_{u}:=\inf\{t\ge 1: \sum_{i=1}^t \log L_{\lambda(Q)}(X_i)\ge u\}$. Under $Q$ we know that the increments $Y_i$ are iid and satisfy, by \eqref{eq:unif-goodlambda}, $\EE_Q[Y_1]=g_Q(\lambda(Q))\ge (1-\varepsilon)I(Q)>0$. Let us now verify the bounded increment hypothesis that Lemma~\ref{lem:bm-hit} gives us. We know that $X_i\in[0,1]$ and $\lambda\in(0,1)$ and thus for every $x\in[0,1]$ we get that,
\[
1-\lambda \le L_\lambda(x)= (1-\lambda)+\frac{\lambda}{m}x \le (1-\lambda)+\frac{\lambda}{m}\le \frac{1}{m}.
\]
Therefore, it follows that $\log (1-\lambda) \le \log L_\lambda(x)\le \log(1/m)$. Now because $\lambda(Q)\in\Lambda_{\Delta,\varepsilon}$ and $\Lambda_{\Delta,\varepsilon}$ is a finite set, the lower bound indeed must be uniformly finite. That is, $a_{\Delta,\varepsilon}:=\min_{\lambda\in\Lambda_{\Delta,\varepsilon}}\log(1-\lambda)>-\infty$. And, $b_m:=\log(1/m)<\infty$. Hence, $Y_1\in[a_{\Delta,\varepsilon}, b_m]$ with probability one under $Q$. Thus, Lemma~\ref{lem:bm-hit} indeed applies. This is now remarkable. Because, the bound we proved in Lemma~\ref{lem:bm-hit} tells us then that for all $\gamma$,
\[
\EE_Q[\tau_{u_\gamma}]\le \frac{u_\gamma + b_m}{\EE_Q[Y_1]}\le \frac{\log(\gamma/w_{\Delta,\varepsilon})+b_m}{(1-\varepsilon)I_Q}.
\]
\noindent Multiplying all this by $I_Q$ and then dividing by $\log\gamma$ gives us that (recall above that we already showed that $C_Q(T^{\mathrm{BM}}_{\gamma})\le \EE_Q[\tau_{u_\gamma}]$),
\[
\sup_{Q\in\mathcal Q_{m+\Delta}} \frac{I_Q\mathcal C_Q(T^{\mathrm{BM}}_\gamma)}{\log\gamma} \le \frac{\log\gamma + \log(1/w_{\Delta,\varepsilon})+b_m}{(1-\varepsilon)\log\gamma}.
\]
\noindent Taking the $\limsup_{\gamma\to\infty}$ thus gives us,
\[
\limsup_{\gamma\to\infty} \sup_{Q\in\mathcal Q_{m+\Delta}}\frac{\mathrm{KL}_{\inf}(Q;m)\mathcal C_Q(T^{\mathrm{BM}}_\gamma)}{\log \gamma} \le \frac{1}{1-\varepsilon}.
\]
\noindent Now, we all know that $\varepsilon\in(0,1/2)$ was arbitrary. Hence, we can let $\varepsilon\downarrow 0$ to get the uniform upper bound as desired, $\le 1$. But! Notice how the corresponding uniform lower bound $\ge 1$ also follows from Corollary~\ref{cor:minimax-cadd-lb} applied with $\mathcal P=\mathcal P_m$ and $\mathcal Q=\mathcal Q_{m+\Delta}$. Therefore, 
\[
\lim_{\gamma\to\infty} \sup_{Q\in\mathcal Q_{m+\Delta}} \frac{\mathrm{KL}_{\inf}(Q;m)\mathcal C_Q(T^{\mathrm{BM}}_\gamma)}{\log\gamma} = 1.
\]
This was exactly our third item we needed to prove. Hence we are done with the proof of this theorem!
\end{proof}

\end{document}